\documentclass[11pt, twoside, leqno]{article}

\usepackage{amsmath}
\usepackage{amssymb}
\usepackage{titletoc}
\usepackage{mathrsfs}
\usepackage{amsthm}
\usepackage{color}
\usepackage{latexsym,amssymb}

\usepackage{indentfirst}
\usepackage{color}
\usepackage{txfonts}

\allowdisplaybreaks

\def\bint{{\ifinner\rlap{\bf\kern.30em--}
\int\else\rlap{\bf\kern.35em--}\int\fi}\ignorespaces}

\def\sbint{{\ifinner\rlap{\bf\kern.32em--}
\hspace{0.078cm}\int\else\rlap{\bf\kern.45em--}\int\fi}\ignorespaces}

\def\red{\color{red}}

\def\rr{{\mathbb R}}
\def\rn{{\mathbb{R}^n}}

\def\nn{{\mathbb N}}
\def\zz{{\mathbb Z}}

\def\fz{\infty }

\def\lf{\left}
\def\r{\right}
\def\ls{\lesssim}
\def\noz{\nonumber}
\def\wz{\widetilde}

\DeclareMathOperator{\supp}{supp}

\def\XXint#1#2#3{{\setbox0=\hbox{$#1{#2#3}{\int}$ }
\vcenter{\hbox{$#2#3$ }}\kern-.6\wd0}}

\def\f{\frac}

\newtheorem{theorem}{Theorem}[section]
\newtheorem{lemma}[theorem]{Lemma}

\newtheorem{assumption}[theorem]{Assumption}
\theoremstyle{definition}
\newtheorem{remark}[theorem]{Remark}
\newtheorem{definition}[theorem]{Definition}
\renewcommand{\appendix}{\par
\setcounter{section}{0}%
\setcounter{subsection}{0}%
\setcounter{subsubsection}{0}%
\gdef\thesection{\@Alph\c@section}%
\gdef\thesubsection{\@Alph\c@section.\@arabic\c@subsection}%
\gdef\theHsection{\@Alph\c@section.}%
\gdef\theHsubsection{\@Alph\c@section.\@arabic\c@subsection}%
\csname appendixmore\endcsname
}

\numberwithin{equation}{section}

\begin{document}

\arraycolsep=1pt

\title{\bf\Large
Boundedness of Fractional Integrals on
Hardy Spaces Associated with Ball
Quasi-Banach Function Spaces
\footnotetext{\hspace{-0.35cm} 2020 {\it
Mathematics Subject Classification}. Primary 42B20;
Secondary 47A30, 42B30, 46E35, 42B25, 42B35.
\endgraf {\it Key words and phrases.}
fractional integral, Hardy-type space,
ball quasi-Banach function space.
\endgraf This project is partially supported by
the National Natural Science Foundation of China (Grant Nos.\
11971058 and 12071197) and the National
Key Research and Development Program of China
(Grant No.\ 2020YFA0712900).}}
\date{}
\author{Yiqun Chen, Hongchao Jia and Dachun Yang\footnote{Corresponding author,
E-mail: \texttt{dcyang@bnu.edu.cn}/{\red June 1, 2022}/Final version.}}
\maketitle

\vspace{-0.8cm}

\begin{center}
\begin{minipage}{13cm}
{\small {\bf Abstract}\quad
Let $X$ be a ball quasi-Banach function space on ${\mathbb R}^n$ and
$H_X({\mathbb R}^n)$ the Hardy space associated with $X$, and let $\alpha\in(0,n)$ and $\beta\in(1,\infty)$.
In this article, assuming that
the (powered) Hardy--Littlewood maximal operator satisfies the Fefferman--Stein
vector-valued maximal inequality on $X$ and is bounded on the associate space of $X$,
the authors prove that the fractional
integral $I_{\alpha}$
can be extended to a bounded linear operator
from $H_X({\mathbb R}^n)$ to $H_{X^{\beta}}({\mathbb R}^n)$
if and only if there exists a positive constant $C$ such that,
for any ball $B\subset \mathbb{R}^n$,
$|B|^{\frac{\alpha}{n}}\leq C
\|\mathbf{1}_B\|_X^{\frac{\beta-1}{\beta}}$,
where $X^{\beta}$ denotes the $\beta$-convexification of $X$.
Moreover, under some different reasonable assumptions
on both $X$ and another ball quasi-Banach function space $Y$,
the authors also consider the mapping property of $I_{\alpha}$
from $H_X({\mathbb R}^n)$ to $H_Y({\mathbb R}^n)$
via using the extrapolation theorem.
All these results have a wide range of applications.
Particularly, when these are
applied, respectively, to
Morrey spaces, mixed-norm Lebesgue spaces,
local generalized Herz spaces, and mixed-norm Herz spaces, all
these results are new. The proofs of these theorems strongly
depend on atomic and molecular characterizations of $H_X({\mathbb R}^n)$.
}
\end{minipage}
\end{center}

\vspace{0.2cm}



\section{Introduction\label{s-intro}}

It is well known that the real-variable theory of Hardy-type spaces
on $\rn$ including the boundedness of fractional integrals
plays an
important role in both harmonic analysis and partial differential equations
(see, for instance, \cite{CW,EMS1970}). Recall
that the classical Hardy space $H^p(\rn)$ with $p\in(0,1]$
was originally introduced by Stein and Weiss
\cite{sw1960} and seminally developed by Fefferman and Stein \cite{FS72}.
Fefferman and Stein \cite{FS72}
established the famous dual theorem, namely, the dual space of
$H^1(\rn)$ is the space $\mathop{\mathrm{BMO}\,}(\rn)$
of functions with bounded mean oscillation,
which was first introduced by John and Nirenberg \cite{JN}.

Let $\alpha\in(0,n)$.
Recall that Hardy and Littlewood \cite{hl1928} introduced the fractional integral
$I_{\alpha}$ which is defined by setting,
for any $f\in L^q(\rn)$ with $q\in [1,\frac{n}{\alpha})$,
and for almost every $x\in \rn$,
\begin{align}\label{cla-I}
I_{\alpha}(f)(x)
:=\int_{\rn}\frac{f(y)}{|x-y|^{n-\alpha}}\,dy.
\end{align}
Hardy and Littlewood \cite{hl1928}
and Sobolev \cite{s1938} established the boundedness of $I_{\alpha}$
from $L^p(\rn)$ to $L^{q}(\rn)$ with $p\in(1,\frac{n}{\alpha})$ and
$\frac{1}{q}:=\frac{1}{p}-\frac{\alpha}{n}$,
which is nowadays called the
Hardy--Littlewood--Sobolev theorem
and plays an important role in both potential theory and partial differential equations;
see, for instance, \cite{clt2019, mv1995, Po1997, R1996}.
From then on, the boundedness of fractional integrals on Hardy-type
spaces has attracted more and more attention.
For instance, Nakai \cite{N17} studied
the boundedness of fractional integrals on generalized Hardy
spaces over spaces of homogeneous type in the sense of Coifman and Weiss,
which are preduals of Campanato spaces with variable growth condition.
Cao et al. \cite{ccy2013} studied fractional integrals on
weighted Orlicz--Hardy spaces.
Recently, Ho \cite{ho2021} studied the Erd\'elyi--Kober fractional integrals
on ball Banach function spaces.
Meanwhile, Huy and Ky \cite{hk2021} obtained the boundedness of
fractional integrals on Musielak--Orlicz Hardy spaces.
We refer the reader to \cite{ans2021, cs2022, chy2021, dll2003,ho2020, N10, Tan2020} for more
studies on the boundedness of fractional integrals on Hardy-type spaces.

On the other hand, Sawano et al. \cite{SHYY}
introduced the ball quasi-Banach function space $X$ and
the related Hardy space $H_X(\rn)$.
Sawano et al. \cite{SHYY} established various real-variable characterizations of
$H_X(\rn)$, respectively,
in terms of various maximal functions,
atoms, molecules, and Lusin-area functions.
Later, Wang et al. \cite{wyy} obtained the boundedness of both Calder\'on--Zygmund and
pseudo-differential operators on both $H_X(\rn)$ and $h_X(\rn)$, the local version of $H_X(\rn)$; Chang et al. \cite{CWYZ2020} established
various Littlewood--Paley function characterizations of $H_X(\rn)$.
Very recently, Yan et al. \cite{yyy20} obtained various intrinsic square
function characterizations of $H_X(\rn)$; Zhang et al. \cite{zhyy2022} introduced
the ball Campanato-type function space $\mathcal{L}_{X,q,s,d}(\rn)$ which proves the
dual space of $H_X(\rn)$, and then established
its Carleson measure characterization.
We refer the reader to \cite{syy,syy2,tyyz,wyy,wyyz,yhyy1,yhyy,zyyw}
for more studies on Hardy-type spaces associated with $X$. However, the boundedness
of fractional integrals on $H_X(\rn)$ is still unknown
even when $X$ is some concrete function spaces, for instance,
the Morrey space, the mixed-norm Lebesgue space,
the local generalized Herz space, and the mixed-norm Herz space.

In this article, we extend the aforementioned Hardy--Littlewood--Sobolev
theorem to Hardy spaces $H_X(\rn)$.
To be precise, assuming that
the (powered) Hardy--Littlewood maximal operator satisfies some Fefferman--Stein
vector-valued maximal inequality on $X$ and is bounded on the associate space of $X$,
we prove that the fractional integral $I_{\alpha}$, with $\alpha\in(0,n)$,
can be extended to a bounded linear operator
from $H_X({\mathbb R}^n)$ to $H_{X^{\beta}}({\mathbb R}^n)$
for some $\beta\in(1,\infty)$
if and only if there exists a positive constant $C$ such that,
for any ball $B\subset \mathbb{R}^n$,
\begin{align}\label{main-align}
|B|^{\frac{\alpha}{n}}\leq C
\|\mathbf{1}_B\|_X^{\frac{\beta-1}{\beta}},
\end{align}
where $X^{\beta}$ denotes the $\beta$-convexification of $X$.
Moreover, under some different reasonable assumptions
on both $X$ and another ball quasi-Banach function space $Y$
recently introduced by Deng et al. \cite{dsl}, we also consider
the mapping property of $I_{\alpha}$
from $H_X({\mathbb R}^n)$ to $H_Y({\mathbb R}^n)$
via using the extrapolation theorem.
All these results have a wide range of applications
and, even when $X$ is the Morrey space, the mixed-norm Lebesgue space,
the local generalized Herz space, and the mixed Herz space, all
these results are new. The proofs of these theorems
strongly depend on atomic and molecular characterizations of $H_X({\mathbb R}^n)$.
To limit the length of this article, the boundedness of fractional integrals on the
Campanato-type function space $\mathcal{L}_{X,q,s,d}(\rn)$,
the dual space of $H_X(\rn)$, will be
studied in the forthcoming article
\cite{cjy2022}.

Indeed, the known proof of fractional integrals on concrete
spaces (see, for instance, \cite{GD2020}) is strongly based on the
various indexes of concrete spaces, and is inapplicable for the Hardy space $H_{X}(\rn)$ here due to the deficiency of the explicit expression of the (quasi-)norm of $X$.
We use two methods to overcome this essential difficulty.
One way is that we creatively introduce
the Hardy space $H_{X^{\beta}}(\rn)$, with some $\beta\in(1,\fz)$,
as the appropriate image space of $I_{\alpha}$ on $H_{X}(\rn)$.
Another way is that we use the extrapolation theorem
to find the appropriate image space and then obtain
the mapping property of $I_{\alpha}$ on $H_X({\mathbb R}^n)$.

To be precise, the remainder of this article is organized as follows.

In Section \ref{s1}, we recall some basic concepts
on the quasi-Banach function space $X$ and the Hardy space $H_X(\rn)$.
At the end of this section, we present two additional assumptions about the
boundedness of the (powered) Hardy--Littlewood maximal operator on $X$
(see Assumptions \ref{assump1} and \ref{assump2} below).

Section \ref{s4} is devoted to obtaining
the boundedness of $I_\alpha$ on $H_X(\rn)$.
To be precise, in Theorem \ref{thm-Ia-02} below,
under Assumptions \ref{assump1} and \ref{assump2} below,
we prove that $I_\alpha$ can be extended to a bounded
linear operator from $H_X(\rn)$ to $H_{X^\beta}(\rn)$, with some $\beta\in(1,\fz)$,
if and only if \eqref{main-align} holds true.
Moreover, motivated by \cite{dsl}, in Theorem \ref{thm-R} below,
we obtain the other mapping property of $I_{\alpha}$
on $H_X({\mathbb R}^n)$ via using the extrapolation theorem,
which can be particularly applied to Orlicz-slice Hardy spaces.

In Section \ref{Appli}, we apply all the above main results to four concrete examples
of ball quasi-Banach function spaces,
namely, the Morrey space $M_r^p(\rn)$,
the mixed-norm Lebesgue space $L^{\vec{p}}(\rn)$,
the local generalized Herz space
$\dot{\mathcal{K}}_{\omega,\mathbf{0}}^{p,q}(\mathbb{R}^{n})$,
and the mixed-norm Herz space $\dot{E}^{\vec{\alpha},\vec{p}}_{\vec{q}}(\rn)$,
respectively. Therefore, all the boundedness of the fractional integral $I_{\alpha}$
on Hardy-type spaces $HM_r^p(\rn)$, $H^{\vec{p}}(\rn)$, $H^p_{w}(\rn)$,
$H\dot{\mathcal{K}}_{\omega,\mathbf{0}}^{p,q}(\mathbb{R}^{n})$,
and $H\dot{E}^{\vec{\alpha},\vec{p}}_{\vec{q}}(\rn)$ are obtained
(see, respectively, Theorems \ref{Th-Morrey}, \ref{Th-Mixed}, \ref{lyq-X-Y},
\ref{Th-LGH1}, \ref{lyq-X-Y-Z}, \ref{Th-M-H}, and \ref{lyq-X} below).

Finally, we make some conventions on notation. Let
$\nn:=\{1,2,\ldots\}$, $\zz_+:=\nn\cup\{0\}$, and $\mathbb{Z}^n_+:=(\mathbb{Z}_+)^n$.
We always denote by $C$ a \emph{positive constant}
which is independent of the main parameters,
but it may vary from line to line.
The symbol $f\lesssim g$ means that $f\le Cg$.
If $f\lesssim g$ and $g\lesssim f$, we then write $f\sim g$.
If $f\le Cg$ and $g=h$ or $g\le h$,
we then write $f\ls g\sim h$
or $f\ls g\ls h$, \emph{rather than} $f\ls g=h$
or $f\ls g\le h$. We use $\mathbf{0}$ to
denote the \emph{origin} of $\rn$.
For any measurable subset $E$ of $\rn$, we denote by $\mathbf{1}_E$ its
characteristic function. Moreover, for any $x\in\rn$ and $r\in(0,\fz)$,
let $B(x,r):=\{y\in\rn:\ |y-x|<r\}$.
Furthermore, for any $\lambda\in(0,\infty)$
and any ball $B(x,r)\subset\rn$ with $x\in\rn$ and
$r\in(0,\fz)$, let $\lambda B(x,r):=B(x,\lambda r)$.
Finally, for any $q\in[1,\infty]$,
we denote by $q'$ its \emph{conjugate exponent},
namely, $\frac{1}{q}+\frac{1}{q'}=1$.

\section{Preliminaries\label{s1}}

In this section, we recall the definitions of ball quasi-Banach function spaces and their related
Hardy spaces $H_X(\rn)$.
In what follows, we use $\mathscr M(\rn)$ to denote the set of
all measurable functions on $\rn$.
For any $x\in\rn$ and $r\in(0,\infty)$, let $B(x,r):=\{y\in\rn:\ |x-y|<r\}$ and
\begin{equation}\label{Eqball}
\mathbb{B}(\rn):=\lf\{B(x,r):\ x\in\rn \text{ and } r\in(0,\infty)\r\}.
\end{equation}
The following concept of ball quasi-Banach function
spaces on $\rn$
is from \cite[Definition 2.2]{SHYY}.

\begin{definition}\label{Debqfs}
Let $X\subset\mathscr{M}(\rn)$ be a quasi-normed linear space
equipped with a quasi-norm $\|\cdot\|_X$ which makes sense for all
measurable functions on $\rn$.
Then $X$ is called a \emph{ball quasi-Banach
function space} on $\rn$ if it satisfies:
\begin{enumerate}
\item[$\mathrm{(i)}$] if $f\in\mathscr{M}(\rn)$, then $\|f\|_{X}=0$ implies
that $f=0$ almost everywhere;
\item[$\mathrm{(ii)}$] if $f, g\in\mathscr{M}(\rn)$, then $|g|\le |f|$ almost
everywhere implies that $\|g\|_X\le\|f\|_X$;
\item[$\mathrm{(iii)}$] if $\{f_m\}_{m\in\nn}\subset\mathscr{M}(\rn)$ and $f\in\mathscr{M}(\rn)$,
then $0\le f_m\uparrow f$ almost everywhere as $m\to\infty$
implies that $\|f_m\|_X\uparrow\|f\|_X$ as $m\to\infty$;
\item[$\mathrm{(iv)}$] $B\in\mathbb{B}(\rn)$ implies
that $\mathbf{1}_B\in X$,
where $\mathbb{B}(\rn)$ is as in \eqref{Eqball}.
\end{enumerate}

Moreover, a ball quasi-Banach function space $X$
is called a
\emph{ball Banach function space} if it satisfies:
\begin{enumerate}
\item[$\mathrm{(v)}$] for any $f,g\in X$,
\begin{equation*}
\|f+g\|_X\le \|f\|_X+\|g\|_X;
\end{equation*}
\item[$\mathrm{(vi)}$] for any ball $B\in \mathbb{B}(\rn)$,
there exists a positive constant $C_{(B)}$,
depending on $B$, such that, for any $f\in X$,
\begin{equation*}
\int_B|f(x)|\,dx\le C_{(B)}\|f\|_X.
\end{equation*}
\end{enumerate}
\end{definition}

\begin{remark}\label{rem-ball-B}
\begin{enumerate}
\item[$\mathrm{(i)}$] Let $X$ be a ball quasi-Banach
function space on $\rn$. By \cite[Remark 2.6(i)]{yhyy1},
we conclude that, for any $f\in\mathscr{M}(\rn)$, $\|f\|_{X}=0$ if and only if $f=0$
almost everywhere.

\item[$\mathrm{(ii)}$] As was mentioned in
\cite[Remark 2.6(ii)]{yhyy1}, we obtain an
equivalent formulation of Definition \ref{Debqfs}
via replacing any ball $B$ by any
bounded measurable set $E$ therein.

\item[$\mathrm{(iii)}$] We should point out that,
in Definition \ref{Debqfs}, if we
replace any ball $B$ by any measurable set $E$ with
finite measure, we obtain the
definition of (quasi-)Banach function spaces which were originally
introduced in \cite[Definitions 1.1 and 1.3]{BS88}. Thus,
a (quasi-)Banach function space
is also a ball (quasi-)Banach function
space and the converse is not necessary to be true.
\item[$\mathrm{(iv)}$] By \cite[Theorem 2]{dfmn2021},
we conclude that both (ii) and (iii) of
Definition \ref{Debqfs} imply that any ball quasi-Banach
function space is complete and the converse is not necessary to be true.
\end{enumerate}
\end{remark}

The associate space $X'$ of any given ball
Banach function space $X$ is defined as follows
(see \cite[Chapter 1, Section 2]{BS88} or \cite[p.\,9]{SHYY}).

\begin{definition}\label{de-X'}
For any given ball quasi-Banach function space $X$, its \emph{associate space}
 (also called the
\emph{K\"othe dual space}) $X'$ is defined by setting
\begin{equation*}
X':=\lf\{f\in\mathscr M(\rn):\ \|f\|_{X'}<\infty\r\},
\end{equation*}
where, for any $f\in X'$,
$$\|f\|_{X'}:=\sup\lf\{\lf\|fg\r\|_{L^1(\rn)}:\ g\in X,\ \|g\|_X=1\r\},$$
and $\|\cdot\|_{X'}$ is called the \emph{associate norm} of $\|\cdot\|_X$.
\end{definition}

\begin{remark}\label{bbf}
From \cite[Proposition 2.3]{SHYY}, we deduce that, if $X$ is a ball
Banach function space, then its associate space $X'$ is also a ball
Banach function space.
\end{remark}

We also recall the concepts of both the convexity and
the concavity of ball quasi-Banach function spaces,
which are a part of \cite[Definition 2.6]{SHYY}.

\begin{definition}\label{Debf}
Let $X$ be a ball quasi-Banach function space and $p\in(0,\infty)$.
\begin{enumerate}
\item[(i)] The \emph{$p$-convexification} $X^p$ of $X$
is defined by setting
$$X^p:=\lf\{f\in\mathscr M(\rn):\ |f|^p\in X\r\}$$
equipped with the \emph{quasi-norm} $\|f\|_{X^p}:=\||f|^p\|_X^{1/p}$ for any $f\in X^p$.

\item[(ii)] The space $X$ is said to be
\emph{concave} if there exists a positive constant
$C$ such that, for any $\{f_k\}_{k\in{\mathbb N}}\subset \mathscr M(\rn)$,
$$\sum_{k=1}^{{\infty}}\|f_k\|_{X}
\le C\left\|\sum_{k=1}^{{\infty}}|f_k|\right\|_{X}.$$
In particular, when $C=1$, $X$ is said to be
\emph{strictly concave}.
\end{enumerate}
\end{definition}

\begin{remark}
It is easy to show that, for any ball quasi-Banach function space $X$ and $p\in(0,\infty)$, the
$p$-convexification $X^p$ of $X$ is also a ball quasi-Banach function space.
\end{remark}

We now present the definition of
$H_X(\rn)$ from \cite[Definition 2.22]{SHYY}.
In what follows, denote by $\mathcal{S}(\rn)$ the space of all
Schwartz functions equipped with the topology
determined by a well-known countable family of norms, and by
$\mathcal{S}'(\rn)$ its topological dual space equipped with the
weak-$*$ topology. For any $N\in\mathbb{N}$ and
$\phi\in\mathcal{S}(\rn)$, let
\begin{equation*}
p_N(\phi):=\sum_{\alpha\in\mathbb{Z}
^{n}_{+},|\alpha|\leq N}
\sup_{x\in\rn}\left\{(1+|x|)^{N+n}|
\partial^{\alpha}\phi(x)|\right\}
\end{equation*}
and
\begin{equation*}
\mathcal{F}_N(\rn):=\left\{\phi\in\mathcal{S}
(\rn):\ p_N(\phi)\in[0,1]\right\},
\end{equation*}
where, for any $\alpha:=(\alpha_1,\ldots,\alpha_n)\in\mathbb{Z}^n_+$,
$|\alpha|:=\alpha_1+\cdots+\alpha_n$ and
$\partial^\alpha:=(\frac{\partial}{\partial
x_1})^{\alpha_1}\cdots(\frac{\partial}{\partial
x_n})^{\alpha_n}$. Moreover, for any $r\in\rr$, we denote by $\lfloor r\rfloor$ (resp., $\lceil r\rceil$) the
\emph{maximal} (resp., \emph{minimal})
\emph{integer not greater} (resp., \emph{less}) \emph{than} $r$.
Also, recall that the \emph{Hardy-Littlewood maximal operator}
$\mathcal{M}$ is defined by setting, for any $f\in
L_{\mathrm{loc}}^1(\rn)$ (the set of all locally integrable functions) and
$x\in\rn$,
\begin{align*}
\mathcal{M}(f)(x):=\sup_{B\ni x}\frac{1}{|B|}\int_{B}
|f(y)|dy,
\end{align*}
where the supremum is taken over all balls $B\in\mathbb{B}(\rn)$
containing $x$.

\begin{definition}\label{2d1}
Let $X$ be a ball quasi-Banach
function space and $N\in\nn$ be sufficiently
large. Then the \emph{Hardy
space} $H_X(\rn)$ is defined to be
the set of all the $f\in\mathcal{S}'(\rn)$
such that
$$
\left\|f\right\|_{H_X(\rn)}
:=\left\|\mathcal{M}_N(f)\right\|_{X}<\infty,
$$
where the \emph{non-tangential
grand maximal function}
$\mathcal{M}_{N}(f)$
of $f\in\mathcal{S}'(\rn)$
is defined by setting, for any $x\in\rn$,
\begin{equation}\label{sec6e1}
\mathcal{M}_{N}(f)(x):=
\sup\left\{|f*\phi_{t}(y)|:\
\phi\in\mathcal{F}_{N}(\rn),\
t\in(0,\infty),\ |x-y|<t\right\}.
\end{equation}
\end{definition}

\begin{remark}
Let all the symbols be the same as in Definition \ref{2d1}.
Assume that there exists an $r\in(0,\fz)$ such that
the Hardy--Littlewood maximal operator $\mathcal{M}$ is bounded on $X^{\frac{1}{r}}$.
If $N\in[\lfloor \frac{n}{r}+1\rfloor,\fz)\cap\mathbb{N}$, then,
by \cite[Theorem 3.1]{SHYY}, we find that the Hardy space
$H_X(\rn)$ is independent of
the choice of $N$.
\end{remark}

Recall that the \emph{Lebesgue space} $L^q(\rn)$ with
$q\in(0,\infty]$
is defined to be the set of all the measurable functions $f$ on $\rn$
such that
$$\|f\|_{L^q(\rn)}:=
\begin{cases}
\displaystyle
\lf[\int_{\rn}|f(x)|^q\, dx\r]^{\frac{1}{q}}
&\text{if}\quad q\in(0,\fz),\\
\displaystyle
\mathop{\mathrm{ess\,sup}}_{x\in\rn}\,|f(x)|
&\text{if}\quad q=\fz
\end{cases}$$
is finite.

In this article, we need the following two mild assumptions
about the
boundedness of the (powered) Hardy--Littlewood maximal operator
on ball quasi-Banach function spaces.

\begin{assumption}\label{assump1}
Let $X$ be a ball quasi-Banach function space. Assume that there
exists a $p_-\in(0,\infty)$ such that,
for any given $p\in(0,p_-)$ and $u\in(1,\infty)$, there exists a
positive constant $C$ such that,
for any $\{f_j\}_{j=1}^\infty\subset\mathscr M(\rn)$,
\begin{equation}\label{ass-02}
\lf\|\lf\{\sum_{j\in\nn}\lf[\mathcal{M}(f_j)\r]^u\r\}^{\frac{1}{u}}
\r\|_{X^{1/p}}
\le C\lf\|\lf(\sum_{j\in\nn}|f_j|^u\r)^{\frac{1}{u}}\r\|_{X^{1/p}}.
\end{equation}
\end{assumption}

\begin{assumption}\label{assump2}
Let $X$ be a ball quasi-Banach function space.
Assume that there exists an $r_0\in(0,\infty)$ and a
$p_0\in(r_0,\infty)$
such that $X^{1/r_0}$ is a ball Banach function space and there exists a positive constant $C$ such that,
for any $f\in(X^{1/r_0})'$,
\begin{align}\label{ass-01}
\lf\|\mathcal{M}^{((p_0/r_0)')}(f)\r\|_{(X^{1/r_0})'}\le
C\lf\|f\r\|_{(X^{1/r_0})'}.
\end{align}
\end{assumption}

\begin{remark}\label{main-remark}
\begin{enumerate}
\item[(i)]Let $X$ be a ball quasi-Banach function
space satisfying Assumption \ref{assump1}
with $p_-\in(0,\infty)$. Let $\beta\in(1,\infty)$. It is easy to prove that \eqref{ass-02} still
holds true
with both $X$ and $p_-$ replaced, respectively, by $X^{\beta}$ and $\beta p_-$.
\item[(ii)]Let $X$ be a ball quasi-Banach function
space satisfying Assumption \ref{assump2}
with both $r_0\in(0,\infty)$ and $p_0\in(r_0,\infty)$.
Let $\beta\in(1,\fz)$. It is easy to prove that \eqref{ass-01} still
holds true
with $X$, $r_0$, and $p_0$ replaced, respectively, by $X^{\beta}$,
$\beta r_0$, and $\beta p_0$.
\item[(iii)]Let both $X$ and $p_-$ satisfy Assumption \ref{assump1}, and
$d\in(0,\fz)$.
Notice that, for any ball $B\in\mathbb{B}(\rn)$ and any
$\beta\in[1,\fz)$,
$\mathbf{1}_{\beta B}\leq (\beta+1)^{\frac{dn}{r}}
[\mathcal{M}(\mathbf{1}_B)]^{\frac{d}{r}}$
with $r\in(0,\min\{d,p_-\})$. By this and Assumption \ref{assump1},
we conclude that, for any $r\in(0,\min\{d,p_-\})$, any
$\beta\in[1,\fz)$,
any sequence $\{B_j\}_{j\in\nn}\subset \mathbb{B}(\rn)$,
and any $\{\lambda_j\}_{j\in\nn}\subset [0,\fz)$,
\begin{align*}
\lf\|\lf(\sum_{j\in\nn}\lambda_j^d\mathbf{1}_{\beta
B_j}\r)^{\frac{1}{d}}\r\|_{X}
&\leq (\beta+1)^{\frac{n}{r}}\lf\|\lf\{\sum_{j\in\nn}
\lf[\mathcal{M}(\lambda_j^{r}\mathbf{1}_{B_j})
\r]^{\frac{d}{r}}\r\}^{\frac{1}{d}}\r\|_{X}\\
&\leq (2\beta)^{\frac{n}{r}}\lf\|\lf\{\sum_{j\in\nn}
\lf[\mathcal{M}(\lambda_j^{r}\mathbf{1}_{B_j})\r]^{\frac{s}{r}}\r\}
^{\frac{r}{d}}
\r\|_{X^{1/r}}^{\frac{1}{r}}\\
&\leq (2\beta)^{\frac{n}{r}}\lf\|\lf(\sum_{j\in\nn}
\lambda_j^d\mathbf{1}_{B_j}\r)^{\frac{r}{d}}\r\|_{X^{1/r}}^{\frac{1}{r}}
=(2\beta)^{\frac{n}{r}}\lf\|\lf(\sum_{j\in\nn}
\lambda_j^d\mathbf{1}_{B_j}\r)^{\frac{1}{d}}\r\|_{X}.
\end{align*}
Particularly, for any $r\in(0,p_-)$ and any ball $B\in
\mathbb{B}(\rn)$, we have
\begin{align}\label{key-c}
\lf\|\mathbf{1}_{\beta B}\r\|_{X}
\leq (2\beta)^{\frac{n}{r}}\|\mathbf{1}_{B}\|_{X}.
\end{align}
\end{enumerate}
\end{remark}

At the end of this section,
we recall the following definitions of both the
$(X,q,s)$-atom and the finite atomic Hardy space
$H_{\mathrm{fin}}^{X,q,s,d}(\rn)$ which are, respectively, from
\cite[Definition 3.5]{SHYY}
and \cite[Definition 1.9]{yyy20}.

\begin{definition}\label{atom}
Let $X$ be a ball quasi-Banach function space, $q\in(1,\infty]$, and
$s\in\zz_+$.
Then a measurable function $a$ on $\rn$ is called an
$(X,q,s)$-\emph{atom}
if there exists a ball $B\in\mathbb{B}(\rn)$ such that
\begin{enumerate}
\item[(i)] $\supp\,(a):=\{x\in\rn:\ a(x)\neq0\}\subset B$;
\item[(ii)]
    $\|a\|_{L^q(\rn)}\le\frac{|B|^{\frac{1}{q}}}{\|\mathbf{1}_B\|_X}$;
\item[(iii)] $\int_{\rn} a(x)x^\gamma\,dx=0$ for any
$\gamma:=(\gamma_1,\ldots,\gamma_n)\in\zz_+^n$ with
$|\gamma|:=\gamma_1+\cdots+\gamma_n\le s$,
here and thereafter, for any $x:=(x_1,\ldots,x_n)\in\rn$,
$x^\gamma:=x_1^{\gamma_1}\cdots x_n^{\gamma_n}$.
\end{enumerate}
\end{definition}

\begin{definition}\label{finatom}
Let both $X$ and $p_-$ satisfy Assumption \ref{assump1}. Assume that
$r_0\in(0,\min\{1,p_-\})$ and $p_0\in(r_0,\infty)$
satisfy Assumption \ref{assump2}.
Let $s\in[\lfloor n(\frac{1}{\min\{1,p_-\}}-1)\rfloor,\infty)\cap\zz_+$,
$d\in(0,r_0]$, and $q\in(\max\{1,p_0\},\infty]$.
The \emph{finite atomic Hardy space}
$H_{\mathrm{fin}}^{X,q,s,d}({{\rr}^n})$,
associated with $X$, is defined to be the set of all finite
linear combinations of $(X,q,s)$-atoms. The quasi-norm
$\|\cdot\|_{H_{\mathrm{fin}}^{X,q,s,d}({{\rr}^n})}$ in
$H_{\mathrm{fin}}^{X,q,s,d}({{\rr}^n})$
is defined by setting, for any $f\in
H_{\mathrm{fin}}^{X,q,s,d}({{\rr}^n})$,
\begin{align*}
\|f\|_{H_{\mathrm{fin}}^{X,q,s,d}({{\rr}^n})}&:=\inf\left\{\left\|
\left[\sum_{j=1}^{N}
\left(\frac{{\lambda}_j}{\|{\mathbf{1}}_{B_j}\|_X}\right)^d
{\mathbf{1}}_{B_j}\right]^{\frac1d}\right\|_{X}
\right\},
\end{align*}
where the infimum is taken over all finite linear combinations of $(X,q,s)$-atoms of
$f$, namely,
$f=\sum_{j=1}^{N}\lambda_ja_j$
with $N\in\nn$,
$\{\lambda_j\}_{j=1}^{N}\subset[0,\infty)$, and $\{a_j\}_{j=1}^{N}$
being
$(X,q,s)$-atoms supported, respectively,
in the balls $\{B_j\}_{j=1}^{N}\subset\mathbb{B}(\rn)$.
\end{definition}

\section{Fractional Integrals on $H_X(\rn)$\label{s4}}

In this section, we study the mapping property of $I_\alpha$ on $H_X(\rn)$.
To be precise, we first give a necessary and
sufficient condition for the boundedness of the
functional integral $I_\alpha$ both from $X$ to $X^\beta$ and
from $H_X(\rn)$ to $H_{X^\beta}(\rn)$
for some $\beta\in(1,\fz)$. Then, using the extrapolation theorem, we obtain the other boundedness
property of $I_\alpha$ from $H_X(\rn)$ to $H_Y(\rn)$.

To give a necessary and sufficient condition for the
boundedness of $I_\alpha$ from $X$ to $X^\beta$,
we need the following assumption
about the boundedness of the Hardy--Littlewood maximal
operator on ball quasi-Banach function spaces.

\begin{assumption}\label{max-assump}
Let $X$ be a ball quasi-Banach space. Assume that there exists
a positive constant $C$ such that, for any $f\in X$,
\begin{align*}
\|\mathcal{M}(f)\|_X\leq C\|f\|_X.
\end{align*}
\end{assumption}

\begin{remark}\label{max-re}
Let both $X$ and $p_-\in(1,\infty)$ satisfy Assumption
\ref{assump1}. Then it is easy to show that $X$ satisfies Assumption \ref{max-assump}.
\end{remark}

\begin{theorem}\label{Th-a-b}
Let $X$ be a ball Banach space satisfying Assumption
\ref{max-assump}, $\beta\in(1,\infty)$,
$\alpha\in(0,n)$, and $I_{\alpha}$ be the same as in
\eqref{cla-I}.
Then $I_\alpha$ is bounded from $X$ to $X^\beta$, namely,
there exists a positive constant $C$ such that, for any $f\in X$,
\begin{align*}
\lf\|I_{\alpha}(f)\r\|_{X^{\beta}}\leq C\|f\|_{X}
\end{align*}
if and only if there exists a positive constant $\wz C$ such that,
for any ball $B\in\mathbb{B}(\rn)$,
\begin{align}\label{main-assump1}
|B|^{\frac{\alpha}{n}}\leq \wz C
\|\mathbf{1}_B\|_X^{\frac{\beta-1}{\beta}}.
\end{align}
\end{theorem}

To prove Theorem \ref{Th-a-b}, we need the following two lemmas.
The following lemma is a corollary of \cite[Lemma 2.2]{is2017}; we omit the details here.

\begin{lemma}\label{is2017}
Let $X$ be a ball Banach function space satisfying Assumption \ref{max-assump}.
Then
\begin{align*}
\sup_{B\in\mathbb{B}(\rn)}
\frac{1}{|B|}\|\mathbf{1}_B\|_X
\|\mathbf{1}_B\|_{X'}<\infty.
\end{align*}
\end{lemma}

\begin{lemma}\label{Iapoint}
Let $\alpha\in(0,n)$ and $I_\alpha$ be the same as in \eqref{cla-I}.
Then there exists a
positive constant $C$ such that, for any ball $B\in\mathbb{B}(\rn)$,
$$I_\alpha(\mathbf{1}_B)\geq
C|B|^{\frac{\alpha}{n}}\mathbf{1}_B
$$
almost everywhere on $\rn$.
\end{lemma}

\begin{proof}
Let all the symbols be the same as in the present lemma.
By the definition of $I_\alpha$, we find that, for any ball
$B:=(x_B,r_B)\in\mathbb{B}(\rn)$ with $x_B\in\rn$ and $r_B\in(0,\infty)$,
and for any $x\in B$,
$$I_\alpha(\mathbf{1}_B)(x)=\int_{\rn}
\frac{\mathbf{1}_B(y)}{|x-y|^{n-\alpha}}\,dy
\gtrsim\frac{1}{r_B^{n-\alpha}}\int_{\rn}\mathbf{1}_B(y)\,dy
\gtrsim|B|^{\frac{\alpha}{n}}.
$$
This finishes the proof of Lemma \ref{Iapoint}.
\end{proof}

Now, we prove Theorem \ref{Th-a-b}.

\begin{proof}[Proof of Theorem \ref{Th-a-b}]
Let all the symbols be the same as in the present theorem.
We first prove the sufficiency. To this end,
let $f\in X$. Then it is easy to show that, for any $R\in(0,\infty)$ and almost every
$x\in\rn$,
\begin{align}\label{*}
|I_\alpha (f)(x)|\leq\int_{|y|\leq R}
\frac{|f(x-y)|}{|y|^{n-\alpha}}\,dy
+\int_{|y|> R}
\frac{|f(x-y)|}{|y|^{n-\alpha}}\,dy
=:F_1(x)+F_2(x).
\end{align}

Now, we estimate $F_1(x)$. Indeed, by both the definition of
$\mathcal{M}(f)$ and $\alpha\in(0,n)$, it is easy to prove that,
for any $R\in(0,\infty)$ and almost every $x\in\rn$,
\begin{align}\label{*1}
F_1(x)&=\sum_{j=0}^{\infty}\int_{2^{-j-1}R<|y|\leq2^{-j}R}
\frac{|f(x-y)|}{|y|^{n-\alpha}}\,dy\\
&\sim\sum_{j=0}^{\infty}\frac{1}{(2^{-j-1}R)^{n-\alpha}}
\int_{2^{-j-1}R<|y|\leq2^{-j}R}|f(x-y)|\,dy\noz\\
&\lesssim R^\alpha\sum_{j=0}^{\infty}
\frac{(2^{-j})^\alpha}{(2^{-j}R)^n}
\int_{|y|\leq2^{-j}R}|f(x-y)|\,dy\lesssim R^\alpha \mathcal{M}(f)(x).\noz
\end{align}
This is a desired estimate of $F_1(x)$.

Next, we consider $F_2(x)$. From the definition of $\|\cdot\|_{X'}$,
we deduce that, for any $R\in(0,\infty)$ and almost every $x\in\rn$,
\begin{align}\label{*2}
F_2(x)&=\sum_{j=0}^{\infty}\int_{2^{j}R<|y|\leq2^{j+1}R}
\frac{|f(x-y)|}{|y|^{n-\alpha}}
\,dy\sim\sum_{j=0}^{\infty}\frac{1}{(2^{j}R)^{n-\alpha}}
\int_{|y-x|\leq2^{j+1}R}|f(y)|\,dy\\
&\lesssim\sum_{j=0}^{\infty}\frac{1}{(2^{j}R)^{n-\alpha}}
\left\|f\right\|_X\left\|\mathbf{1}_{B(x,2^{j+1}R)}\right\|_{X'}.\noz
\end{align}
Moreover, by Lemma \ref{is2017} and \eqref{main-assump1},
we conclude that, for any ball $B\in\mathbb{B}(\rn)$,
\begin{align*}
\|\mathbf{1}_B\|_{X'}\lesssim
\frac{|B|}{\|\mathbf{1}_B\|_X}\lesssim
|B|^{\frac{\alpha\beta}{n(1-\beta)}+1},
\end{align*}
which, combined with \eqref{*2}, $\alpha\in(0,n)$, and $\beta\in(1,\infty)$, further
implies that,
for any $R\in(0,\infty)$ and almost every $x\in\rn$,
\begin{align}\label{*2'}
F_2(x)\ls\sum_{j=0}^{\infty}
2^{\frac{\alpha j}{1-\beta}}R^{\frac{\alpha}{1-\beta}}\|f\|_X
\lesssim R^{\frac{\alpha}{1-\beta}}\|f\|_X.
\end{align}
This is a desired conclusion of $F_2(x)$.

All together, combining \eqref{*}, \eqref{*1}, and \eqref{*2'} with
$R:=[\frac{\|f\|_X}{\mathcal{M}(f)(x)}]^{\frac{\beta-1}{\alpha\beta}}$,
we obtain, for almost every $x\in\rn$,
$$|I_\alpha (f)(x)|\lesssim\|f\|_X^{1-\frac{1}{\beta}}
[\mathcal{M}(f)(x)]^{\frac{1}{\beta}}.$$
By this, the definition of $\|\cdot\|_{X^\beta}$, and Assumption
\ref{max-assump}, we find that
\begin{align*}
\|I_\alpha (f)\|_{X^\beta}\lesssim
\|f\|_X^{1-\frac{1}{\beta}}\lf\|[\mathcal{M}(f)]^{\frac{1}{\beta}}
\r\|_{X^\beta}
\sim
\|f\|_X^{1-\frac{1}{\beta}}\lf\|\mathcal{M}(f)\r\|_{X}^{\frac{1}{\beta}}
\lesssim\|f\|_X.
\end{align*}
This finishes the proof of the sufficiency.

Now, we show the necessity. Indeed, from Lemma \ref{Iapoint} and the boundedness of
$I_\alpha$ from $X$ to $X^\beta$, we deduce that, for any ball $B\in\mathbb{B}(\rn)$,
$$|B|^{\frac{\alpha}{n}}\|\mathbf{1}_B\|_{X^\beta}
\lesssim\|I_\alpha(\mathbf{1}_B)\|_{X^\beta}
\lesssim\|\mathbf{1}_B\|_X.
$$
Combining this and the definition of $\|\cdot\|_{X^\beta}$, we conclude that,
for any ball $B\in\mathbb{B}(\rn)$,
$$|B|^{\frac{\alpha}{n}}
\lesssim\|\mathbf{1}_B\|_X\|\mathbf{1}_B\|_{X^\beta}^{-1}
\sim\|\mathbf{1}_B\|
_X^{\frac{\beta-1}{\beta}}.
$$
This finishes the proof of the necessity, and hence of
Theorem \ref{Th-a-b}.
\end{proof}

As an application of Theorem \ref{Th-a-b},
we give the boundedness of fractional maximal operators from $X$ to
$X^\beta$.
In what follows, for any $\alpha\in(0,n)$,
the \emph{central fractional maximal operator $\mathcal{M}_\alpha$}
is defined by setting, for any $f\in L^1_{\mathrm{loc}}(\rn)$ and
$x\in\rn$,
\begin{align}\label{de-Ma}
\mathcal{M}_\alpha(f)(x):=\sup_{r\in(0,\infty)}
\frac{1}{r^{n-\alpha}}\int_{|y|\leq r}|f(x-y)|\,dy.
\end{align}

\begin{theorem}\label{coclassic}
Let $X$ be a ball Banach space satisfying Assumption
\ref{max-assump}, $\alpha\in(0,n)$, and $\beta\in(1,\infty)$.
Then $\mathcal{M}_\alpha$ is bounded from $X$ to $X^\beta$, namely,
there exists a positive constant $C$ such that, for any $f\in X$,
\begin{align*}
\lf\|\mathcal{M}_{\alpha}(f)\r\|_{X^{\beta}}\leq C\|f\|_{X}
\end{align*}
if and only if there exists a positive constant $\wz C$ such that,
for any ball $B\in\mathbb{B}(\rn)$,
\begin{align*}
|B|^{\frac{\alpha}{n}}\leq \wz C
\|\mathbf{1}_B\|_X^{\frac{\beta-1}{\beta}}.
\end{align*}
\end{theorem}

\begin{proof}
Let all the symbols be the same as in the present theorem.
We first prove the sufficiency.
To this end, let $I_{\alpha}$ be the same as in \eqref{cla-I} and $f\in
L^1_{\mathrm{loc}}(\rn)$.
It is well known that, for any $r\in(0,\fz)$ and $x\in\rn$,
$\mathcal{M}_\alpha(f)(x)
\leq I_{\alpha}(|f|)(x)$ (see, for instance, \cite[p.\,138, (3.2.4)]{LDY2007}).
Combining this, Definition \ref{Debqfs}, and Theorem \ref{Th-a-b}, we find that, for any $f\in L^1_{\mathrm{loc}}(\rn)$,
\begin{align*}
\lf\|\mathcal{M}_\alpha(f)\r\|_{X^\beta}\leq
\lf\|I_{\alpha}(|f|)\r\|_{X^\beta}\lesssim
\|f\|_X.
\end{align*}
This finishes the proof of the sufficiency.

Now, we prove the necessity.
Notice that, for any ball $B:=(x_B,r_B)\in\mathbb{B}(\rn)$ with $x_B\in\rn$ and $r_B\in(0,\infty)$, and for any $x\in B$,
\begin{align*}
\mathcal{M}_\alpha(\mathbf{1}_B)(x)&=\sup_{r\in(0,\infty)}
\frac{1}{r^{n-\alpha}}\int_{|y|\leq r}|\mathbf{1}_B(x-y)|\,dy\\
&\gtrsim\frac{1}{r_B^{n-\alpha}}\int_{\rn}\mathbf{1}_B(y)\,dy
\gtrsim|B|^{\frac{\alpha}{n}}.
\end{align*}
Using this and repeating the proof of the necessity of Theorem \ref{Th-a-b} with
$I_\alpha$ replaced by $\mathcal{M}_\alpha$, we complete the
proof of the necessity, and hence of Theorem \ref{coclassic}.
\end{proof}

Next, we establish the boundedness of $I_\alpha$ from $H_X(\rn)$ to
$H_{X^\beta}(\rn)$ with $\alpha\in(0,n)$
under the assumption that $X$ has an absolutely continuous quasi-norm.

\begin{theorem}\label{thm-Ia-02}
Let $X$, $p_-$, $r_0\in(0,\min\{\frac{1}{\beta},p_-\})$, and
$p_0\in(r_0,\infty)$ satisfy both
Assumptions \ref{assump1} and \ref{assump2} with $\beta\in(1,\infty)$.
Further assume that $X$ has an absolutely continuous quasi-norm.
Let $\alpha\in(0,n)$
and $I_\alpha$ be the same as in \eqref{cla-I}.
Then $I_\alpha$ can be extended to a unique bounded linear
operator, still denoted by $I_\alpha$, from $H_X(\rn)$ to
$H_{X^\beta}(\rn)$, namely,
there exists a positive constant $C$ such that, for any $f\in
H_X(\rn)$,
\begin{align*}
\lf\|I_{\alpha}(f)\r\|_{H_{X^{\beta}}(\rn)}\leq C\|f\|_{H_X(\rn)}
\end{align*}
if and only if there exists a positive constant $\wz C$ such that,
for any ball $B\in\mathbb{B}(\rn)$,
\begin{align*}
|B|^{\frac{\alpha}{n}}\leq \wz C
\|\mathbf{1}_B\|_X^{\frac{\beta-1}{\beta}}.
\end{align*}
\end{theorem}

To prove this theorem, we need both the
atomic and the molecular characterizations of $H_X(\rn)$.
The following lemma is a consequence of both
\cite[Theorem 3.7]{SHYY} and its proof.

\begin{lemma}\label{atom-shyy}
Let both $X$ and $p_-$ satisfy Assumption \ref{assump1}. Let
$s\in[\lfloor n(\frac{1}{\min\{1,p_-\}}-1)\rfloor,\infty)\cap\zz_+$
and $d\in(0,\min\{1,p_-\})$.
Then, for any $f\in H_X(\rn)$, there exists a sequence
$\{a_{j}\}_{j\in\nn}$
of $(X,\fz,s)$-atoms supported, respectively, in
$\{B_{j}\}_{j\in\nn}\subset \mathbb{B}(\rn)$,
and a sequence $\{\lambda_{j}\}_{j\in\nn}\subset[0,\fz)$ such that
$$
f=\sum_{j\in\nn}\lambda_j a_{j}
$$
in $\mathcal{S}'(\rn)$, and
$$
\left\|\left\{\sum_{j\in\mathbb{N}}\left[\frac{\lambda_{j}}
{\|\mathbf{1}_{B_{j}}\|_X}\right]^{d}
\mathbf{1}_{B_{j}}\right\}^{\frac{1}{d}}\right\|_X
\lesssim\|f\|_{H_X(\rn)},
$$
where the implicit positive constant is independent of $f$, but may depend on
$d$.
\end{lemma}

\begin{definition}\label{def-mol}
Let $X$ be a ball quasi-Banach function space, $q\in(1,\infty]$,
$s\in\zz_+$, and $\tau\in(0,\fz)$. Then a measurable function $M$ on
$\rn$ is called an
\emph{$(X,q,s,\tau)$-molecule} centered at a ball
$B\in \mathbb{B}(\rn)$ if
\begin{enumerate}
\item[(i)] for any $j\in\zz_+$,
\begin{align*}
\lf\|M\mathbf{1}_{L_j}\r\|_{L^q(\rn)}
\leq 2^{-j\tau}\frac{|B|^{\frac{1}{q}}}{\|\mathbf{1}_{B}\|_{X}},
\end{align*}
where $L_0:=B$ and, for any $j\in\nn$, $L_j:=2^{j}B\setminus
2^{j-1}B$;
\item[(ii)] $\int_{\rn} M(x)x^\gamma\,dx=0$ for any
$\gamma\in\zz_+^n$ with $|\gamma|\le s$.
\end{enumerate}
\end{definition}

\begin{remark}\label{mole-atom}
Let all the symbols be the same as in Definition \ref{def-mol}.
It is easy to show that any $(X,q,s)$-atom is also an
$(X,q,s,\tau)$-molecule.
\end{remark}

From both \cite[Theorem 3.9]{SHYY} and its proof, we can
deduce the following lemma; we omit the details.

\begin{lemma}\label{2l1}
Let $X$, $q$, $s$, and $d$ be the same as in Definition \ref{finatom}, and
$\tau\in (n[\frac{1}{\min\{1,p_{-}\}}-\frac{1}{q}],\fz)$. Then $f\in H_X(\rn)$ if and
only if
there exists a sequence $\{M_j\}_{j\in\nn}$ of
$(X,q,s,\tau)$-molecules centered,
respectively, at balls $\{B_j\}_{j\in\nn}\subset\mathbb{B}(\rn)$ and
$\{{\lambda}_j\}_{j\in\nn}\subset[0,\infty)$ satisfying
$$\left\|\left\{\sum_{j\in\nn}
\left(\frac{{\lambda}_j}{\|{\mathbf{1}}_{B_j}\|_X}\right)^d
{\mathbf{1}}_{B_j}\right\}^{\frac1d}\right\|_{X}<\fz$$
and $f=\sum_{j\in\nn}{\lambda}_jM_j$ in $\mathcal{S}'(\rn)$.
Moreover,
$$\lf\|f\r\|_{H_X(\rn)}\sim\left\|\left\{\sum_{j\in\nn}
\left(\frac{{\lambda}_j}{\|{\mathbf{1}}_{B_j}\|_X}\right)^d
{\mathbf{1}}_{B_j}\right\}^{\frac1d}\right\|_{X},$$
where the positive equivalence constants are independent of $f$.
\end{lemma}

We also need the following two lemmas.
The first lemma is well known (see, for instance,
\cite[p.\,119]{EMS1970}).

\begin{lemma}\label{fractional}
Let $\alpha\in(0,n)$ and $I_{\alpha}$ be the same as in \eqref{cla-I}.
Let $p\in[1,\frac{n}{\alpha})$ and $q\in(1,\infty)$ with
$\frac{1}{q}:=\frac{1}{p}-\frac{\alpha}{n}$.
\begin{enumerate}
\item[\rm (i)]
For any $f\in L^p(\rn)$, $I_\alpha(f)(x)$ is well defined for
almost every $x\in\rn$;

\item[\rm (ii)]
If $p\in(1,\frac{n}{\alpha})$, then $I_\alpha$ is bounded from
$L^p(\rn)$
to $L^{q}(\rn)$, namely, there exists a
positive constant $C$ such that, for any $f\in L^p(\rn)$,
$$\lf\|I_\alpha(f)\r\|_{L^{q}(\rn)}\leq
C\|f\|_{L^p(\rn)}.$$
\end{enumerate}
\end{lemma}

\begin{lemma}\label{lemma-atm}
Let $X$ be a ball quasi-Banach space and
$s\in[n-1,\infty)\cap\zz_+$.
Let $\beta\in(1,\infty)$, $\alpha\in(0,n)$, and $I_\alpha$ be the same as in
\eqref{cla-I}.
Assume that $p\in(1,\frac{n}{\alpha})$.
If there exists a positive constant $C$ such that,
for any ball $\wz B\in\mathbb{B}(\rn)$,
\begin{align}\label{assum-01}
\lf|\wz B\r|^{\frac{\alpha}{n}}\leq C
\lf\|\mathbf{1}_{\wz B}\r\|_X^{\frac{\beta-1}{\beta}},
\end{align}
then, for any $(X,p,s)$-atom $a$ supported in a ball $B\in\mathbb{B}(\rn)$, $I_\alpha(a)$
is an $(X^\beta,q,s-n+1,\tau)$-molecule
centered at $B$, up to a harmless constant
multiple,
where $\frac{1}{q}:=\frac{1}{p}-\frac{\alpha}{n}$
and $\tau\in(0,n+s+1-\alpha-\frac{n}{q}]$.
\end{lemma}

\begin{proof}
Let all the symbols be the same as in the present lemma.
Let $a$ be an $(X,p,s)$-atom supported in a ball $B\in\mathbb{B}(\rn)$.
We first show that $I_\alpha(a)$ satisfies Definition
\ref{def-mol}(i), namely, for any $j\in\zz_+$,
\begin{align}\label{sizecondition}
\left\|I_\alpha(a)\mathbf{1}_{L_j}\right\|_{L^q(\rn)}
\lesssim2^{-\tau
j}\frac{|B|^{\frac{1}{q}}}{\left\|\mathbf{1}_B\right\|_{X^\beta}},
\end{align}
where $L_0:=B$ and $L_j:=2^{j}B\setminus 2^{j-1}B$ for any
$j\in\nn$.

Indeed, by Lemma \ref{fractional}(ii),
$\frac{1}{q}=\frac{1}{p}-\frac{\alpha}{n}$,
Definition \ref{atom}(ii), \eqref{assum-01}, and
the definition of $\|\cdot\|_{X^{\beta}}$, we conclude that
\begin{align}\label{j=0}
\lf\|I_\alpha (a)\mathbf{1}_{2B}\r\|_{L^q(\rn)}
&\lesssim \|a\|_{L^p(\rn)}
\lesssim \frac{|B|^{\frac{1}{p}}}{\|\mathbf{1}_{B}\|_X}
\sim \frac{|B|^{\frac{1}{q}+\frac{\alpha}{n}}}{\|\mathbf{1}_{B}\|_X}
\ls\frac{|B|^{\frac{1}{q}}}{\|\mathbf{1}_{B}\|_{X}^{\beta}}
\sim\frac{|B|^{\frac{1}{q}}}{\|\mathbf{1}_{B}\|_{X^{\beta}}}.
\end{align}
Combining this, $\|I_\alpha (a)\mathbf{1}_{L_0}\|_{L^q(\rn)}\leq\|I_\alpha
(a)\mathbf{1}_{2B}\|_{L^q(\rn)}$, and $\|I_\alpha
(a)\mathbf{1}_{L_1}\|_{L^q(\rn)}\leq\|I_\alpha (a)\mathbf{1}_{2B}\|_{L^q(\rn)}$,
we find that \eqref{sizecondition} holds true when $j\in\{0,1\}$.

Now, we prove \eqref{sizecondition} for any $j\in[2,\infty)\cap\nn$.
To this end, let $x_B$ denote the center of $B$ and $r_B$ its
radius. Then, from
Definition \ref{atom}(iii), the Taylor remainder theorem,
the estimation of $J_2$ in \cite[p.\,106]{Lu},
the H\"older inequality, Definition \ref{atom}(ii),
$\frac{1}{q}=\frac{1}{p}-\frac{\alpha}{n}$,
and an argument similar to that used in the estimation of \eqref{j=0},
we deduce that, for any $j\in[2,\infty)\cap\nn$,
$x\in 2^{j}B\setminus 2^{j-1}B$, and $y\in B$,
there exists a $\widetilde{y}\in B$ such that
\begin{align}\label{es-Ia}
\lf|I_\alpha (a)(x)\r|
&=\lf|\int_{B}\frac{a(y)}{|x-y|^{n-\alpha}}\,dy\r|\\
&=\lf|\int_{B}\lf[\frac{1}{|x-y|^{n-\alpha}}
-\sum_{\{\gamma\in\zz_+^n:\ |\gamma|\leq s\}}
\frac{\partial_y^{\gamma}(\frac{1}{|x-y|^{n-\alpha}})|_{y=x_B}}
{\gamma!}(y-x_B)^{\gamma}\r]a(y)\,dy\r|\noz\\
&\sim\lf|\int_{B}\lf[\sum_{\{\gamma\in\zz_+^n:\ |\gamma|=s+1\}}
\partial_{y}^{\gamma}\lf(\frac{1}{|x-y|^{n-\alpha}}\r)\big|_{y=\wz
y}(y-x_B)^{\gamma}\r]a(y)\,dy\r|\noz\\
&\lesssim\int_{B}\frac{|y-x_B|^{s+1}|a(y)|}{|x-\wz
y|^{n+s+1-\alpha}}\,dy
\lesssim\frac{|B|^{\frac{1}{p'}}
r_B^{s+1}}{|x-x_B|^{n+s+1-\alpha}}\|a\|_{L^{p}(\rn)}\noz\\
&\ls \frac{r^{\frac{n}{p'}+s+1}}{(2^{j}r)^{n+s+1-\alpha}}
\frac{|B|^{\frac{1}{p}}}{\|\mathbf{1}_{B}\|_{X}}
\lesssim 2^{-j(n+s+1-\alpha)}r^{-\frac{n}{q}}
\frac{|B|^{\frac{1}{q}}}{\|\mathbf{1}_{B}\|_{X^{\beta}}}\noz,
\end{align}
where $\frac{1}{p}+\frac{1}{p'}=1$.
By this and $\tau\in(0,n+s+1-\alpha-\frac{n}{q}]$,
we obtain, for any $j\in[2,\infty)\cap\nn$,
\begin{align*}
\lf\|I_\alpha(a)\mathbf{1}_{L_j}\r\|_{L^q(\rn)}
&\ls 2^{-j(n+s+1-\alpha)}|L_j|^{\frac{1}{q}}r^{-\frac{n}{q}}
\frac{|B|^{\frac{1}{q}}}{\|\mathbf{1}_{B}\|_{X^{\beta}}}\\
&\lesssim 2^{-j(n+s+1-\alpha-\frac{n}{q})}
\frac{|B|^{\frac{1}{q}}}{\|\mathbf{1}_{B}\|_{X^{\beta}}}
\lesssim
2^{-j\tau}\frac{|B|^{\frac{1}{q}}}{\|\mathbf{1}_{B}\|_{X^{\beta}}},
\end{align*}
which implies that \eqref{sizecondition} holds true for any
$j\in[2,\infty)\cap\nn$.

Next, we show that $I_\alpha(a)$ satisfies Definition
\ref{def-mol}(ii),
namely, $\int_{\rn} I_\alpha(a)(x)x^\gamma\,dx=0$ for any
$\gamma\in\zz_+^n$ with $|\gamma|\le s-n+1$.
Without loss of generality, we may assume that $a$ is
supported in the ball $B:=B(\mathbf{0},r)\in\mathbb{B}(\rn)$,
where $r\in(0,\infty)$.
We first show $I_\alpha(a)(\cdot)|\cdot|^\gamma\in
L^1(\rn)$ for any
$\gamma\in\zz_+^n$ with $|\gamma|\leq s-n+1$.
Obviously, for any
$\gamma\in\zz_+^n$ with $|\gamma|\leq s-n+1$,
\begin{align}\label{J1J2}
\|I_\alpha(a)|\cdot|^\gamma\|_{L^1(\rn)}&\leq
\int_{2B}|I_\alpha(a)(x)||x|^\gamma\,dx+
\int_{\rn\setminus2B}\cdots\\
&=:\mathrm{J_1}+\mathrm{J_2}.\noz
\end{align}
We first estimate $\mathrm{J_1}$.
Indeed, using the H\"older inequality and Lemma \ref{fractional}(ii),
we conclude that
\begin{align}\label{J1}
\mathrm{J_1}\leq\|I_\alpha(a)\|_{L^q(\rn)}|2B|^{1-\frac{1}{q}}
\lesssim\|a\|_{L^p(\rn)}|2B|^{1-\frac{1}{q}}
<\infty.
\end{align}
As for $\mathrm{J_2}$, notice that
\begin{align*}
\mathrm{J_2}=\int_{\rn\setminus2B}\lf|\int_{B}\frac{a(y)}
{|x-y|^{n-\alpha}}\,dy\r||x|^\gamma\,dx.
\end{align*}
By this, the H\"older inequality,
$|\gamma|\leq s-n+1$, and $\alpha\in(0,n)$, similarly to the estimation of \eqref{es-Ia}, we obtain
\begin{align}\label{J2}
\mathrm{J_2}&\lesssim\int_{\rn\setminus
2B}\lf|\int_{B}\frac{|a(y)||y|^{s+1}}{|x|^{n-\alpha+s+1-\gamma}}\,dy\r|\,dx\\
&\lesssim\|a\|_{L^p(\rn)}|B|^{\frac{s+1}{n}+1-\frac{1}{p}}
\int_{\rn\setminus 2B}\frac{1}{|x|^{n-\alpha+s+1-|\gamma|}}\,dx
<\infty.\noz
\end{align}
Combining \eqref{J1J2}, \eqref{J1}, and \eqref{J2}, we conclude that
$I_\alpha(a)(\cdot)|\cdot|^\gamma\in L^1(\rn)$ with $|\gamma|\leq s-n+1$.
Then, repeating the proof in \cite[p.\,104]{tw80}
(see also the proof of \cite[p.\,105,
Theorem 3.1]{Lu}), we find that $I_\alpha(a)$ satisfies Definition \ref{def-mol}(ii), which completes the proof of Lemma \ref{lemma-atm}.
\end{proof}

Now, via borrowing some ideas from
the proof of \cite[Theorem 1.3]{hk2021}, we prove Theorem \ref{thm-Ia-02}.

\begin{proof}[Proof of Theorem \ref{thm-Ia-02}]
Let all the symbols be the same as in the present theorem.
We first show the sufficiency.
To this end, let $d\in(0,r_0]$ and $s\in\zz_+$ be such that
\begin{align}\label{ass-s-s}
s\geq\max\lf\{\lf\lfloor\frac{n}
{\min\{1,p_-\}}-n\r\rfloor,
\lf\lfloor\frac{n}{\min\{1,\beta
p_-\}}-1\r\rfloor,
\lf\lfloor\frac{n}{\min\{1,\beta
p_-\}}-n+\alpha\r\rfloor\r\}.
\end{align}
Moreover, by the assumption that $X$ has an absolutely continuous
quasi-norm
and \cite[Proposition 3.13]{zhyy2022}, we conclude that
$H_{\text{fin}}^{X, \infty, s, d}(\rn)\cap C(\rn)$
is dense in $H_X(\rn)$, here and thereafter, $C(\rn)$ denotes
the set of all continuous functions on $\rn$. From this and a standard density argument, we
deduce that,
to prove the present theorem, it suffices to show that,
for any $f\in  H_{\text{fin}}^{X, \infty, s, d}(\rn)\cap C(\rn)$,
\begin{align*}
\|I_\alpha (f)\|_{H_{X^\beta}(\rn)}
\lesssim\|f\|_{H_X(\rn)}.
\end{align*}

Indeed, by Definition \ref{finatom}, we conclude that, for any
$f\in  H_{\text{fin}}^{X, \infty, s, d}(\rn)\cap C(\rn)$, there exists an
$m \in \mathbb{N}$, a sequence $\{a_{j}\}_{j=1}^{m}$ of
$(X, \infty, s)$-atoms supported, respectively, in the balls
$\{B_{j}\}_{j=1}^{m}\subset\mathbb{B}(\rn)$, and
a sequence $\{\lambda_{j}\}_{j=1}^{m} \subset[0, \infty)$
such that $f=\sum_{j=1}^{m} \lambda_{j} a_{j}$ and
\begin{align}\label{finsize1}
\left\|\left\{\sum_{j=1}^{m}\left[\frac{\lambda_{j}}
{\left\|\mathbf{1}_{B_{j}}\right\|_X}\right]^{d}
\mathbf{1}_{B_{j}}\right\}^{\frac{1}{d}}\right\|_X
\lesssim\|f\|_{H_{\text {fin }}^{X, \infty, s, d}(\rn)}.
\end{align}

Moreover, notice that, for any $j\in\mathbb{N}$,
$a_j$ is also an $(X,p,s)$-atom with
$p\in([\frac{1}{\max\{1,\beta p_0\}}
+\frac{\alpha}{n}]^{-1},\frac{n}{\alpha})\cap(1,\frac{n}{\alpha})$.
Using \eqref{ass-s-s}, we find that we can choose a
$\tau\in(n[\frac{1}{\min\{1,\beta
p_{-}\}}-\frac{1}{q}],n+s+1-\alpha-\frac{n}{q}]$, where
$\frac{1}{q}:=\frac{1}{p}-\frac{\alpha}{n}$ is
such that $q\in(\max\{1,\beta p_0\},\infty)$.
By this and Lemma \ref{lemma-atm}, we conclude that, for any $j\in\nn$,
$I_\alpha(a_j)$	is an $(X^\beta,q,s-n+1,\tau)$-molecule
centered at $B_j$
up to a harmless constant multiple.
From both (i) and (ii) of Remark \ref{main-remark}, we infer that $X^\beta$, $\beta p_-$,
$\beta r_0$, $\beta p_0$, $q$, $s-n+1$, $\beta d$, and $\tau$
satisfy all the assumptions of Lemma \ref{2l1}.
Combining this with Lemma \ref{2l1}, the fact that, for any $i\in\{1,\ldots,m\}$,
$$
\left\|\left\{\sum_{j=1}^{m}
\left[\frac{\lambda_{j}^\beta}{\|\mathbf{1}_{B_{j}}\|_X}\right]^{d}
\mathbf{1}_{B_{j}}\right\}^{\frac{1}{d}}\right\|_X
\geq \lambda_i^\beta,
$$
the definition of $\|\cdot\|_{X^{\beta}}$,
$\beta\in(1,\infty)$, \eqref{finsize1}, and
\cite[Theorem 1.10]{yyy20} (see also \cite[Lemma 3.12]{zhyy2022}), we conclude
that, for any $f\in  H_{\text{fin}}^{X, \infty, s, d}(\rn)\cap C(\rn)$,
\begin{align}\label{IaHx}
\left\|I_\alpha(f)\right\|_{H_{X^\beta}(\rn)}
&\lesssim\left\|\left\{\sum_{j=1}^{m}
\left[\frac{\lambda_j}{\|\mathbf{1}_{B_j}\|_{X^\beta}}\right]^{d\beta}
\mathbf{1}_{B_j}\right\}^{\frac{1}{d\beta}}\right\|_{X^\beta}\\
&\sim\left\|\left\{\sum_{j=1}^{m}\left[\frac{\lambda_j^\beta}
{\|\mathbf{1}_{B_j}\|_X}\right]^{d}\mathbf{1}_{B_j}\right\}
^{\frac{1}{d}}\right\|_X
\left\|\left\{\sum_{j=1}^{m}\left[\frac{\lambda_j^\beta}
{\|\mathbf{1}_{B_j}\|_X}\right]^{d}\mathbf{1}_{B_j}
\right\}^{\frac{1}{d}}\right\|_X^{\frac{1-\beta}{\beta}}\noz\\
&\lesssim\left\|\left\{\sum_{j=1}^{m}\left[\frac{\lambda_j^\beta
\min\{\lambda_1^{1-\beta},\ldots,\lambda_m^{1-\beta}\}}
{\|\mathbf{1}_{B_j}\|_X}\right]^{d}\mathbf{1}_{B_j}\right\}
^{\frac{1}{d}}\right\|_X\noz\\
&\lesssim\left\|\left\{\sum_{j=1}^{m}\left[\frac{\lambda_j}
{\|\mathbf{1}_{B_j}\|_X}\right]^{d}\mathbf{1}_{B_j}\right\}
^{\frac{1}{d}}\right\|_X
\lesssim\|f\|_{H_{\text {fin }}^{X, \infty, s,
d}(\rn)}\sim\left\|f\right\|_{H_X(\rn)}.\noz
\end{align}
This finishes the proof of the sufficiency.

Now, we prove the necessity. To this end, let $a_{0}$ be an $(X,
\infty, s)$-atom,
associated with the unit ball $B_{0}:=B(\mathbf{0},1) \subset \mathbb{R}^{n}$,
such that $I_{\alpha} (a_{0})$ is continuous and not identically zero.
Thus, we find that there exists a ball
$\wz{B_{0}}:=B(\wz{x_0},\wz{r_0}) \subset B_{0}$,
with $\wz{x_0}\in\rn$ and $\wz{r_0}\in(0,\infty)$,
and a positive constant $c_{0}$ such that
$\left|I_{\alpha} (a_{0})(x)\right| \geq c_{0}$ for any $x \in
B_{0}$.
Let $B:=B(x_{B}, r_{B})$ be an arbitrary ball of $\mathbb{R}^{n}$
with $x_{B} \in \mathbb{R}^{n}$ and
$r_{B} \in(0, \infty)$. For any $x\in\rn$, let
$L_{B}(x):=x_{B}+r_{B} x$. Then $L_B$ is an affine transformation on
$\mathbb{R}^{n}$.
Now, we consider $a_{B}(\cdot):=a_{0}(L_{B}^{-1}(\cdot))$.
By a change of variables, we conclude that, for any
$x\in\wz{B}:=L_{B}(\wz{B_{0}})$,
\begin{align*}
\left|I_{\alpha} (a_{B})(x)\right|
&=\lf|\int_{\rn}\frac{a_B(y)}{|x-y|^{n-\alpha}}\,dy\r|
=|B|^{\frac{\alpha-n}{n}}\lf|\int_{\rn}\frac{a_B(y)}
{|L_{B}^{-1}(x)-L_{B}^{-1}(y)|^{n-\alpha}}\,dy\r|\\
&=|B|^{\frac{\alpha}{n}}\lf|\int_{\rn}\frac{a_0(y)}
{|L_{B}^{-1}(x)-y|^{n-\alpha}}\,dy\r|\\
&=|B|^{\frac{\alpha}{n}}\left|I_{\alpha}
(a_0)\left(L_{B}^{-1}(x)\right)\right|
\gtrsim|B|^{\frac{\alpha}{n}}.
\end{align*}
From this, the definition of $\|\cdot\|_{H_{X^\beta}(\rn)}$,
the estimate that $\mathcal{M}_N(I_\alpha(a_B))
\geq|I_\alpha(a_B)|$
almost everywhere, and Definition \ref{Debqfs}(ii),
we deduce that
\begin{align}\label{eqIa}
\left\|I_{\alpha} (a_{B})\right\|_{H_{X^\beta}(\rn)}
=\left\|\mathcal{M}_N(I_\alpha (a_B))(x)\right\|_{X^\beta}
\gtrsim\lf\|I_\alpha(a_B)\mathbf{1}_{\wz{B_0}}\r\|_{X^\beta}
\gtrsim |B|^{\frac{\alpha}
{n}}\left\|\mathbf{1}_{\wz{B_0}}\right\|_{X^\beta},
\end{align}
where $\mathcal{M}_N(I_\alpha (a_B))$ is as in \eqref{sec6e1}
with $f$ therein replaced by $I_\alpha (a_B)$.
Moreover, by the boundedness of $I_{\alpha}$ from
$H_X\left(\mathbb{R}^{n}\right)$ to
$H_{X^\beta}\left(\mathbb{R}^{n}\right)$,
we have $\left\|I_{\alpha} (a_{B})\right\|_{H_{X^\beta}(\rn)}
\lesssim\left\|a_{B}\right\|_{H_X(\rn)}$. Using this and
\eqref{eqIa}, we conclude that
\begin{align}\label{eq1}
|B|^{\frac{\alpha}{n}}\left\|\mathbf{1}_{\wz{B_0}}\right\|_{X^\beta}
\lesssim\left\|a_{B}\right\|_{H_{X}(\rn)}.
\end{align}
Let $\wz a_B:=\|\mathbf{1}_B\|_X^{-1}\|a_B\|_{L^\infty(\rn)}^{-1}a_B$.
Then we have $\|\wz
a_B\|_{L^\infty(\rn)}\leq\frac{1}{\|\mathbf{1}_B\|_X}$.
Combining this and Definition \ref{atom},
we find that $\wz a_B$ is an $(X,\infty,s)$-atom supported in $B$.
By this and \cite[Theorem 3.6]{SHYY}
(see also \cite[Lemma 3.11]{zhyy2022}), we find that
$$\|\wz
a_B\|_{H_X(\rn)}
\sim\|\wz a_B\|_{H_{\text {fin}}^{X, \infty, s,
d}(\rn)}
\lesssim\lf\|\frac{\mathbf{1}_B}{\|\mathbf{1}_B\|_X}\r\|_X\sim1.
$$
Combining this and $\|a_B\|_{L^\infty(\rn)}=
\|a_0\|_{L^\infty(\rn)}$, we conclude that
\begin{align}\label{eq2}
\|a_B\|_{H_X(\rn)}\lesssim\|\mathbf{1}_B\|_X\|a_0\|_{L^\infty(\rn)}
\lesssim\|\mathbf{1}_B\|_X.
\end{align}
From $B\subset\frac{2}{\wz r_0}\wz B$ and \eqref{key-c}, we deduce that
\begin{align*}
\lf\|\mathbf{1}_{B}\r\|_{X^\beta}\leq
\lf\|\mathbf{1}_{\frac{2}{\wz r_0}\wz B}\r\|_{X^\beta}\leq
\lf(\frac{4}{\wz r_0}\r)^{\frac{2n}{p_-}}
\lf\|\mathbf{1}_{\wz B}\r\|_{X^\beta}.
\end{align*}
Therefore, combining this, \eqref{eq1}, \eqref{eq2},
and the definition of $\|\cdot\|_{X^{\beta}}$, we have
\begin{align*}
|B|^{\frac{\alpha}{n}}
\ls\|a_{B}\|_{H_{X}(\rn)}\lf\|\mathbf{1}_{\wz{B}}\r\|_{X^\beta}^{-1}
\ls\|\mathbf{1}_{B}\|_{X}\lf\|\mathbf{1}_{B}\r\|_{X^{\beta}}^{-1}
\sim\|\mathbf{1}_{B}\|_{X}^{\frac{\beta-1}{\beta}}.
\end{align*}
This finishes the proof of the necessity, and hence of Theorem
\ref{thm-Ia-02}.
\end{proof}

Next, we weaken the assumption that $X$
has an absolutely continuous quasi-norm into a weaker assumption
which is applicable to Morrey spaces.

\begin{theorem}\label{thm-Ia-01}
Let both $X$ and $p_-$ satisfy Assumption \ref{assump1}. Let
$\beta\in(1,\fz)$, $\alpha\in(0,n)$, and $I_\alpha$ be the same as in \eqref{cla-I}. Assume that
$X$ satisfies Assumption \ref{assump2}
with both $r_0\in(0,\min\{\frac{1}{\beta},
p_-,\frac{n}{n+\alpha}\})$ and
$p_0\in(r_0,\frac{r_0}{1-r_0})$.
Then $I_\alpha$ can be extended to a bounded linear operator,
still denoted by $I_\alpha$, from $H_X(\rn)$ to $H_{X^\beta}(\rn)$,
namely,
there exists a positive constant $C$ such that, for any $f\in
H_X(\rn)$,
\begin{align*}
\lf\|I_{\alpha}(f)\r\|_{H_{X^{\beta}}(\rn)}\leq C\|f\|_{H_X(\rn)}
\end{align*}
if and only if there exists a positive constant $\wz C$ such that,
for any ball $B\in\mathbb{B}(\rn)$,
\begin{align*}
|B|^{\frac{\alpha}{n}}\leq \wz C
\|\mathbf{1}_B\|_X^{\frac{\beta-1}{\beta}}.
\end{align*}
\end{theorem}

To prove this theorem, we first recall the
concept of the following weights (see, for instance,
\cite[pp.\,21 and 139]{LDY2007}).

\begin{definition}
Let $p\in[1,\infty)$ and $w$ be a nonnegative locally integrable
function on $\rn$. Then
$w$ is called an \emph{$A_{p}(\rn)$ weight}, denoted by $w\in A_p(\rn)$, if, when $p\in(1,\infty)$,
\begin{align*}
[w]_{A_{p}\left(\mathbb{R}^{n}\right)}
:=\sup _{B \subset \mathbb{R}^{n}}\frac{1}{|B|}
\lf[\int_{B}w(x)\,dx\r]\left\{\frac{1}{|B|}
\int_{B}[w(x)]^{-\frac{1}{p-1}} \,dx\right\}^{p-1}<\infty,
\end{align*}
and
\begin{align*}
[w]_{A_{1}\left(\mathbb{R}^{n}\right)}
:=\sup _{B \subset \mathbb{R}^{n}}\frac{1}{|B|}
\lf[\int_{B}w(x)\,dx\r]\left\{\mathop{\mathrm{ess\,sup}}_{x\in\rn}
[w(x)]^{-1}\right\}<\infty,
\end{align*}
where the suprema are taken over all balls $B \in\mathbb{B}(\rn)$.
Moreover, the \emph{class $A_\infty(\rn)$} is defined by setting
$$
A_\infty(\rn):=\bigcup_{p\in[1,\infty)}A_p(\rn).
$$
\end{definition}

\begin{definition}
Let $1<p<q<\infty$ and $w$ be a nonnegative locally integrable
function on $\rn$. Then $w$ is called an \emph{$A_{p,q}(\rn)$ weigh}t,
denoted by $w\in A_{p,q}(\rn)$, if
$$
[w]_{A_{p,
q}\left(\mathbb{R}^{n}\right)}:=\sup _{B
\subset
\mathbb{R}^{n}}\left\{\frac{1}{|B|}
\int_{B}[w(x)]^{q} \,d
x\right\}^{\frac{1}{q}}\left\{\frac{1}{|B|}
\int_{B}[w(x)]^{-p'} \,d
x\right\}^{\frac{1}{p}}<\infty,
$$
where $\frac{1}{p}+\frac{1}{p'}=1$.
\end{definition}

Now, we recall the definition of weighted Lebesgue spaces (see, for instance, \cite[Definition 7.1.3]{G1}).

\begin{definition}
Let $p\in(0,\infty)$ and $w\in A_\infty(\rn)$.
The \emph{weighted Lebesgue space} $L_{w}^{p}(\mathbb{R}^{n})$ is
defined to be the set of all the measurable
functions $f$ on $\rn$ such that
\begin{align*}
\lf\|f\r\|_{L^p_w(\rn)}=\lf[\int_{\rn}|f(x)|^pw(x)\,dx\r]
^{\frac{1}{p}}<\infty.
\end{align*}
\end{definition}

The following lemma is just a special case of \cite[Theorem
1.2]{ccy2013}.
In what follows, for any $p\in(0,\fz)$ and any
nonnegative locally integrable function $w\in A_\infty(\rn)$,
the \emph{weighted Hardy space} $H^p_w(\rn)$ is
defined as in Definition \ref{2d1} with $X:=L^p_{w}(\rn)$.

\begin{lemma}\label{I-reasonable}
Let $\alpha \in(0,n)$, $p,q\in(0,1]$, and $p_{1},q_{1} \in(1,
\infty)$ be
such that $\frac{1}{p}-\frac{1}{q}=\frac{\alpha}{n}
=\frac{1}{p_{1}}-\frac{1}{q_{1}}$.
Then, for any $w^{\frac{1}{q_{1}}} \in A_{p_{1}, q_{1}}
\left(\mathbb{R}^{n}\right), I_{\alpha}$
can be extended to a bounded linear operator
from $H_{w^p}^{p}\left(\mathbb{R}^{n}\right)$
to $H_{w^q}^{q}\left(\mathbb{R}^{n}\right)$
and, moreover, there exists a positive constant $C$,
independent of $w$ itself but depending on $[w]_{A_{p,q}(\rn)}$, such that, for any
$f\in H_{w^p}^{p}\left(\mathbb{R}^{n}\right)$,
$$\lf\|I_\alpha(f)\r\|_{H_{w^q}^{q}\left(\mathbb{R}^{n}\right)}
\leq C\lf\|f\r\|_{H_{w^p}^{p}\left(\mathbb{R}^{n}\right)}.
$$
\end{lemma}

The following lemma is from the proof of
\cite[Lemma 4.7]{CWYZ2020}.

\begin{lemma}\label{weighted-embed}
Let $X$ be a ball quasi-Banach function
space satisfying Assumption \ref{assump2}
with $r_0\in(0,\infty)$ and $p_0\in(r_0,\infty)$.
Then, for any $\epsilon\in(1-\frac{r_0}{p_0},1)$,
$X$ embeds continuously into $L_w^{r_0}(\rn)$
with $w:=[\mathcal{M}(\mathbf{1}_{B(\mathbf{0},1)})]^\epsilon$.
\end{lemma}

Now, we prove Theorem \ref{thm-Ia-01}.

\begin{proof}[Proof of Theorem \ref{thm-Ia-01}]
Let all the symbols be the same as in the present theorem.
We first show the sufficiency. To this end, let $d\in(0,r_0]$ and $s\in\zz_+$ be
such that
\begin{align}\label{ass-s}
s\geq\max\lf\{\lf\lfloor\frac{n}
{\min\{1,p_-\}}-n\r\rfloor,
\lf\lfloor\frac{n}{\min\{1,\beta
p_-\}}-1\r\rfloor,
\lf\lfloor\frac{n}{\min\{1,\beta
p_-\}}-n+\alpha\r\rfloor\r\}.
\end{align}
Since $X$, $p_-$, $d$, and $s$
satisfy all the assumptions of Lemma
\ref{atom-shyy}, from Lemma \ref{I-reasonable}, we infer that,
for any $f\in H_X(\rn)$, there exists
a sequence $\{a_j\}_{j\in\mathbb{N}}$
of $(X,\infty,s)$-atoms supported, respectively, in the balls
$\{B_j\}_{j\in\mathbb{N}}\subset\mathbb{B}(\rn)$, and a sequence
$\{\lambda_j\}_{j\in\mathbb{N}}\subset[0,\infty)$
such that $f=\sum_{j\in\mathbb{N}}\lambda_ja_j$ in
$\mathcal{S}'(\rn)$, and
\begin{align}\label{finsize}
\left\|\left\{\sum_{j\in\mathbb{N}}\left[\frac{\lambda_{j}}
{\|\mathbf{1}_{B_{j}}\|_X}\right]^{d}
\mathbf{1}_{B_{j}}\right\}^{\frac{1}{d}}\right\|_X
\lesssim\|f\|_{H_X(\rn)}.
\end{align}

Notice that $p_0\in(r_0,\frac{r_0}{1-r_0})$ implies
$r_0>1-\frac{r_0}{p_0}$,
and hence there exists a $\gamma\in(0,1)$
such that $r_0\gamma\in(1-\frac{r_0}{p_0},1)$.
From this and Lemma \ref{weighted-embed}, we
deduce that $X$ embeds continuously
into $L_{w^{r_0}}^{r_0}(\rn)$ with
$w:=[\mathcal{M}(\mathbf{1}_{B(\mathbf{0},1)})]^\gamma$.
Then, using \cite[(3.10)]{wyy} with $Y$ therein replaced by
$L_{w^{r_0}}^{r_0}(\rn)$,
we obtain
\begin{align}\label{con-01}
f=\sum_{j\in\mathbb{N}}\lambda_ja_j
\end{align}
in $H_{w^{r_0}}^{r_0}(\rn)$.
Moreover, by \cite[Theorem 7.7\,(1)]{Duo01}, we
find $w\in A_1(\rn)$, which,
together with \cite[Proposition
7.2\,(1)]{Duo01}, further implies that $w\in
A_1(\rn)\subset A_{1+\frac{q_1}{p_1}}(\rn)$ for any $p_{1},q_{1} \in(1, \infty)$
such that
$\frac{1}{p_{1}}-\frac{1}{q_{1}}=\frac{\alpha}{n}$.
From this and \cite[Theorem 3.2.2]{LDY2007}, we deduce that
$w^{\frac{1}{q_1}}\in
A_{p_1,q_1}(\rn)$.
Then, using this and Lemma
\ref{I-reasonable}, we
conclude that $I_\alpha$
is bounded from $H_{w^{r_0}}^{r_0}(\rn)$
to $H_{w^{\wz{r}_0}}^{\wz{r}_0}(\rn)$
with $\frac{1}{\wz{r_0}}:=
\frac{1}{r_0}-\frac{\alpha}{n}$. Thus, from this and \eqref{con-01},
we deduce that
\begin{align*}
I_\alpha (f)=\sum_{j\in\mathbb{N}}\lambda_jI_\alpha(a_j)
\end{align*}
in $H_{w_{\wz{r}_0}}^{\wz{r}_0}(\rn)$.

Moreover, notice that, for any $j\in\mathbb{N}$,
$a_j$ is also an $(X,p,s)$-atom with
$p\in([\frac{1}{\max\{1,\beta p_0\}}
+\frac{\alpha}{n}]^{-1},\frac{n}{\alpha})\cap(1,\frac{n}{\alpha})$.
Using \eqref{ass-s}, we know that we can choose a
$\tau\in(n[\frac{1}{\min\{1,\beta
p_{-}\}}-\frac{1}{q}],n+s+1-\alpha-\frac{n}{q}]$, where
$\frac{1}{q}:=\frac{1}{p}-\frac{\alpha}{n}$ and hence
$q\in(\max\{1,\beta p_0\},\infty)$.
By Lemma \ref{lemma-atm}, we find that, for any $j\in\nn$,
$I_\alpha(a_j)$	is an $(X^\beta,q,s-n+1,\tau)$-molecule
centered at $B_j$
up to a harmless constant multiple.
From both (i) and (ii) of Remark \ref{main-remark},
we infer that $X^\beta$, $\beta p_-$,
$\beta r_0$, $\beta p_0$, $q$, $s-n+1$, $\beta d$, and $\tau$
satisfy all the assumptions of Lemma \ref{2l1}.
Combining this with Lemma \ref{2l1}, the estimate that, for any $i\in\nn$,
$$
\left\|\left\{\sum_{j\in\nn}
\left[\frac{\lambda_{j}^\beta}{\|\mathbf{1}_{B_{j}}\|_X}\right]^{d}
\mathbf{1}_{B_{j}}\right\}^{\frac{1}{d}}\right\|_X
\geq \lambda_i^\beta,
$$
the definition of $\|\cdot\|_{X^{\beta}}$,
$\beta\in(1,\infty)$, and \eqref{finsize},
similarly to the estimation of \eqref{IaHx}, we conclude that
\begin{align*}
\left\|I_\alpha(f)\right\|_{H_{X^\beta}(\rn)}
\lesssim\left\|\left\{\sum_{j\in\nn}\left[\frac{\lambda_j}
{\|\mathbf{1}_{B_j}\|_X}\right]^{d}\mathbf{1}_{B_j}\right\}
^{\frac{1}{d}}\right\|_X
\lesssim\left\|f\right\|_{H_X(\rn)}.
\end{align*}
This finishes the proof of the sufficiency.
The proof of the necessity is just a repetition of that of Theorem
\ref{thm-Ia-02},
which completes the proof of Theorem \ref{thm-Ia-01}.
\end{proof}

In the remainder of this section, via borrowing some ideas from \cite{dsl} and
using the extrapolation theorem for Hardy-type spaces,
we establish the boundedness of the
functional integral $I_\alpha$ from $H_X(\rn)$ to $H_Y(\rn)$.

\begin{theorem}\label{thm-R}
Let both $X$ and $Y$ be ball quasi-Banach function spaces.
Let $\alpha\in(0,n)$ and $I_{\alpha}$ be the same as in \eqref{cla-I}.
Let $0<p_0<q_0\leq 1$ be such that
$\frac{1}{q_0}=\frac{1}{p_0}-\frac{\alpha}{n}$.
Assume that $X^{\frac{1}{p_0}}$ and $Y^{\frac{1}{q_0}}$
are ball Banach function spaces, and
\begin{align}\label{thm-R-02}
\lf(Y^{\frac{1}{q_0}}\r)'=\lf(\lf(X^{\frac{1}{p_0}}\r)'\r)
^\frac{p_0}{q_0}.
\end{align}
Further assume that the Hardy--Littlewood maximal
operator $\mathcal{M}$ is bounded on $(Y^{\frac{1}{q_0}})'$.
Then $I_{\alpha}$ can be extended to a bounded linear operator,
still denoted by $I_\alpha$, from $H_X(\rn)$ to $H_{Y}(\rn)$,
namely, there exists a positive constant $C$ such that, for any $f\in
H_X(\rn)$,
\begin{align}\label{thm-R-05}
\lf\|I_{\alpha}(f)\r\|_{H_{Y}(\rn)}\leq C\|f\|_{H_X(\rn)}.
\end{align}
Moreover, the extension is unique if $X$ has an absolutely continuous
quasi-norm.
\end{theorem}

To prove this theorem, we present some technical lemmas.
The following two lemmas are just \cite[Lemmas 2.6 and 2.5]{wyyz}, respectively.

\begin{lemma}\label{lem-R-04}
Every ball Banach function space X coincides with its
second associate space $X''$, namely, $f\in X$ if and only if $f\in
X''$ and, moreover, for any $f\in X$,
$\|f\|_X=\|f\|_{X''}$.
\end{lemma}

\begin{lemma}\label{lem-R-03}
Let $X$ be a ball Banach function space
and $X'$ its associate space. Then, for any $f\in X$ and $g\in X'$,
$fg$ is integrable and
$$
\int_{\rn}|f(x)g(x)|\,dx\leq \|f\|_X\|g\|_{X'}.
$$
\end{lemma}

As a simple application of Lemma \ref{lem-R-04}, Definition \ref{de-X'}, and Lemma
\ref{lem-R-03}, we immediately obtain the following conclusion;
we omit the details.

\begin{lemma}\label{lem-R-02}
Let $X$ be a ball Banach function space with the associate space
$X'$. Then,
for any $f\in X$,
\begin{align*}
\|f\|_X=\sup_{\{g\in X':\ \|g\|_{X'}=1\}}\lf|\int_{\rn}f(x)g(x)\,dx\r|.
\end{align*}
\end{lemma}

Now, we prove Theorem \ref{thm-R}.

\begin{proof}[Proof of Theorem \ref{thm-R}]
Let all the symbols be the same as in the present theorem.
Moreover, for any $h\in (Y^{\frac{1}{q_0}})'$ and $x\in\rn$, let
\begin{align*}
\mathcal{R}h(x):=\sum_{k\in\zz_+}\frac{\mathcal{M}^kh(x)}
{2^k\|\mathcal{M}\|_{(Y^{{1}/{q_0}})'\to(Y^{{1}/{q_0}})'}^k},
\end{align*}
where $\mathcal{M}^0h:=|h|$ and, for any $k\in\nn$,
$\mathcal{M}^kh:=\mathcal{M}\circ\cdots\circ \mathcal{M}$ denotes
the $k$-th iteration of $\mathcal{M}$.
Then, by the definition of $\mathcal{R}h$, the assumption of the
boundedness of $\mathcal{M}$
on $(Y^{\frac{1}{q_0}})'$, and the Levi lemma, we easily find that
$\mathcal{R}h$ has the following properties (see, for instance,
\cite{dsl}):
\begin{enumerate}
\item[(i)]
for almost every $x\in\rn$, $|h(x)|\leq \mathcal{R}h(x)$;

\item[(ii)]
$\|\mathcal{R}h\|_{(Y^{{1}/{q_0}})'}\leq
2\|h\|_{(Y^{{1}/{q_0}})'}$;

\item[(iii)]
$\mathcal{M}(\mathcal{R}h)\leq
2\|\mathcal{M}\|_{(Y^{{1}/{q_0}})'\to(Y^{{1}/{q_0}})'}
\mathcal{R}h$, namely, $\mathcal{R}h\in A_{1}(\rn)$.
\end{enumerate}

To show \eqref{thm-R-05}, we first claim
that $H_X(\rn)$ embeds continuously into
$H^{p_0}_{[\mathcal{R}h]^{p_0/q_0}}(\rn)$.
Indeed, from Lemma \ref{lem-R-03}, the above item (ii), and \eqref{thm-R-02}, we
deduce that
\begin{align}\label{thm-R-03}
\|f\|_{H^{p_0}_{[\mathcal{R}h]^{p_0/q_0}}(\rn)}
&=\lf\{\int_{\rn}\lf[\mathcal{M}_N(f)(x)\r]^{p_0}
[\mathcal{R}h(x)]^{\frac{p_0}{q_0}}\,dx\r\}^{\frac{1}{p_0}}\\
&\leq\lf\|\lf[\mathcal{M}_N(f)\r]^{p_0}\r\|_{X^{{1}/{p_0}}}^{\frac{1}{p_0}}
\lf\|(\mathcal{R}h)^{\frac{p_0}{q_0}}\r\|_{(X^{{1}/{p_0}})'}^{\frac{1}{p_0}}\noz\\
&=\|\mathcal{M}_N(f)\|_{X}
\|\mathcal{R}h\|_{[(X^{{1}/{p_0}})']^{{p_0}/{q_0}}}^\frac{1}{q_0}
=\|f\|_{H_X(\rn)}\|\mathcal{R}h\|_{(Y^{1/q_0})'}^\frac{1}{q_0}\noz\\
&\ls\|f\|_{H_X(\rn)}\|h\|_{(Y^{1/q_0})'}^\frac{1}{q_0},\noz
\end{align}
where $\mathcal{M}_N(f)$ is the same as in \eqref{sec6e1}.
This shows that the above claim holds true.
Moreover, by Lemma \ref{I-reasonable}, we conclude that $I_\alpha$
can be extended to a bounded linear operator
from $H^{p_0}_{[\mathcal{R}h]^{p_0/q_0}}(\rn)$
to $H_{\mathcal{R}h}^{q_0}(\rn)$ and hence is well defined on $H_X(\rn)$
due to the above claim.

Now, we show \eqref{thm-R-05}.
Indeed, using the definitions of both $H_Y(\rn)$ and $Y^{\frac{1}{q_0}}$,
Lemma \ref{lem-R-04}, and the definition of $(Y^{\frac{1}{q_0}})''$, we conclude that
\begin{align}\label{thm-R-06}
\|I_\alpha(f)\|_{H_Y(\rn)}
&=\lf\|\mathcal{M}_N(I_\alpha(f))\r\|_{Y}\\
&=\lf\|\lf[\mathcal{M}_N(I_\alpha(f))\r]^{q_0}\r\|_{Y^{1/q_0}}^{\frac{1}{q_0}}
=\lf\|\lf[\mathcal{M}_N(I_\alpha(f))\r]^{q_0}\r\|_{(Y^{1/q_0})''}^{\frac{1}{q_0}}\noz\\
&=\sup_{\|h\|_{(Y^{1/q_0})'}=1}\lf\{\int_{\rn}\lf[\mathcal{M}_N
(I_\alpha(f))(x)\r]^{q_0}|h(x)|\,dx\r\}^{\frac{1}{q_0}}.\noz
\end{align}
Moreover, from the above items (i) and (iii), Lemma \ref{I-reasonable}
with $w:=(\mathcal{R}h)^{\frac{1}{q_0}}$, and \eqref{thm-R-03}, we infer that
\begin{align*}
&\lf\{\int_{\rn}\lf[\mathcal{M}_N(I_\alpha(f))(x)\r]^{q_0}|h(x)|\,dx\r\}^{\frac{1}{q_0}}\\
&\quad\leq\lf\{\int_{\rn}\lf[\mathcal{M}_N(I_\alpha(f))(x)
\r]^{q_0}\mathcal{R}h(x)\,dx\r\}^{\frac{1}{q_0}}
=\|I_{\alpha}(f)\|_{H_{\mathcal{R}h}^{q_0}(\rn)}\noz\\
&\quad\ls\|f\|_{H_{(\mathcal{R}h)^{p_0/q_0}}^{p_0}(\rn)}
\ls\|f\|_{H_X(\rn)}\|h\|_{(Y^{1/q_0})'}^{\frac{1}{q_0}}.\noz
\end{align*}
By this and \eqref{thm-R-06}, we find that
\begin{align*}
\|I_{\alpha}(f)\|_{H_{Y}(\rn)}\ls \|f\|_{H_{X}(\rn)}.
\end{align*}
This shows that \eqref{thm-R-05} holds true.

To end the proof of the present theorem, we now need to show that,
if $X$ has an absolutely continuous quasi-norm, then
the extension of $I_{\alpha}$ is unique. Indeed,
using the assumption that $X$ has an absolutely continuous quasi-norm
and \cite[Proposition 3.13]{zhyy2022}, we conclude that
$H_{\text{fin}}^{X, \infty, s, d}(\rn)\cap C(\rn)$
is dense in $H_X(\rn)$, where both $s$ and $d$ are the same as in Definition
\ref{finatom}. From this, \eqref{thm-R-05},
and a standard density argument, we deduce the desired result.
This finishes the proof of Theorem \ref{thm-R}.
\end{proof}

\begin{remark}
Let all the symbols be the same as in Theorem \ref{thm-R}.
\begin{enumerate}
\item[(i)]	
We point out that \eqref{thm-R-02} was first introduced by Deng et al. \cite{dsl},
in which they used \eqref{thm-R-02} to obtain the boundedness of fractional integrals with rough kernels and their commutators on ball Banach function spaces.
	
\item[(ii)]
The extrapolation theorem on weighted Lebesgue spaces was originally
established by Rubio de Francia
in \cite{Ru1982} and has been widely used in the study of
sublinear operators on Hardy-type spaces. We refer the reader
to \cite{chy2021,Ho2020,Ho} for more studies on this.
\item[(iii)]
Let $\alpha\in(0,n)$, $p\in(0,\frac{n}{\alpha})$, and
$q\in(0,\fz)$ be such that $\frac{1}{q}=\frac{1}{p}-\frac{\alpha}{n}$.
It is easy to show that Theorem \ref{thm-R} when both
$X:=L^p(\rn)$ and $Y:=L^q(\rn)$
coincides with \cite[p.\,102,\ Theorem (4.1)]{tw80}.

\item[(iv)]
Let $t\in(0,\fz)$ and $p,q\in(0,\frac{n}{\alpha})$.
Recall that the \emph{amalgam space} $(L^p,l^q)(\rn)$ is defined
to be the set of all the
measurable functions $f$ on $\rn$ such that
$$
\|f\|_{(L^p,l^q)(\rn)}:=\lf\{\int_{\rn}\lf[\int_{B(x,t)}|f(y)|^p\,dy
\r]^{\frac{q}{p}}\,dx\r\}^{\frac{1}{q}}<\fz.
$$
Moreover, let the \emph{Hardy-amalgam space}
$(H^p,l^q)(\rn):=H_{(L^p,l^q)(\rn)}(\rn)$ be as in Definition
\ref{2d1}
with $X:=(L^p,l^q)(\rn)$. Then Theorem \ref{thm-R} with both
$X:=(L^p,l^q)(\rn)$ and $Y:=(L^r,l^s)(\rn)$
is a generalization of \cite[Theorem 7]{chy2021}, where $r,s\in(0,\fz)$ are such that $\frac{1}{p}-\frac{1}{r}=
\frac{\alpha}{n}=\frac{1}{q}-\frac{1}{s}$.

\item[(v)]
Recall that a function $\Phi:\ [0,\infty)\to[0,\infty)$
is called an \emph{Orlicz function} if it is
non-decreasing and satisfies $\Phi(0)= 0$,
$\Phi(t)>0$ whenever $t\in(0,\infty)$,
and $\lim_{t\to\infty}\Phi(t)=\infty$. Then $\Phi$ is said to be
of \emph{positive lower} (resp., \emph{positive upper}) \emph{type}
$p\in(0,\infty)$
if there exists a positive constant $C_{(p)}$,
depending on $p$, such that, for any $t\in[0,\infty)$
and $s\in(0,1)$ [resp., $s\in [1,\infty)$],
\begin{equation*}
\Phi(st)\le C_{(p)}s^p \Phi(t).
\end{equation*}
For any Orlicz function $\Phi$ with positive lower
type $p_{\Phi}^-$ and positive upper type $p_{\Phi}^+$,
the \emph{Orlicz space $L^\Phi(\rn)$} is defined
to be the set of all the measurable functions $f$ on $\rn$ such that
$$\|f\|_{L^\Phi(\rn)}:=\inf\lf\{\lambda\in(0,\infty):\
\int_{\rn}\Phi\lf(\frac{|f(x)|}{\lambda}\r)\,dx\le1\r\}<\infty.$$
Moreover, let $t,q\in(0,\infty)$.
The \emph{Orlicz-slice space} $(E_\Phi^q)_t(\rn)$
is defined to be the set of all the measurable functions $f$ on $\rn$
such that
$$
\|f\|_{(E_\Phi^q)_t(\rn)}
:=\lf\{\int_{\rn}\lf[\frac{\|f\mathbf{1}_{B(x,t)}\|_{L^\Phi(\rn)}}
{\|\mathbf{1}_{B(x,t)}\|_{L^\Phi(\rn)}}\r]^q\,dx\r\}^{\frac{1}{q}}<\infty.
$$
Furthermore, let the \emph{Orlicz-slice Hardy space}
$(HE_\Phi^q)_t(\rn):=H_{(E_\Phi^q)_t(\rn)}(\rn)$ be as in
Definition \ref{2d1}
with $X:=(E_\Phi^q)_t(\rn)$. Assume that $q\in(0,\frac{n}{\alpha})$
and $\Phi$ is of lower type $p_1\in(0,\fz)$
and of upper type $q_1\in(0,\frac{n}{\alpha})$.
From \cite[Lemma 3.1]{Ho} and its proof, we deduce that
there exists an Orlicz function $\Psi$ with lower type $p_2$
and upper type $q_2$ such that
$\frac{1}{p_1}-\frac{1}{p_2}=\frac{\alpha}{n}=\frac{1}{q_1}-\frac{1}{q_2}$
and, moreover, \eqref{thm-R-02} holds true with $X:=(E_\Phi^q)_t(\rn)$,
$Y:=(E_\Psi^r)_t(\rn)$, $p_0\in (0,\min\{p_1,q_1,q,1,\frac{n}{n+\alpha}\})$,
and $\frac{1}{q_0}:=\frac{1}{p_0}-\frac{\alpha}{n}$,
where $\frac{1}{r}:=\frac{1}{q}-\frac{\alpha}{n}$.
Thus, Theorem \ref{thm-R} with both $X:=(E_\Phi^q)_t(\rn)$ and $Y:=(E_\Psi^r)_t(\rn)$ coincides with \cite[Theorem 3.1]{Ho}.
\end{enumerate}
\end{remark}

\section{Applications\label{Appli}}

In this section, we apply our main results to
four concrete examples of ball quasi-Banach
function spaces, namely, Morrey spaces
(Subsection \ref{Morrey}), mixed-norm
Lebesgue spaces (Subsection
\ref{Mixed-Norm}), Local
Generalized Herz spaces (Subsection
\ref{H-H}), and Mixed Herz spaces
(Subsection \ref{M-H}).

\subsection{Morrey Spaces\label{Morrey}}

Recall that, due to the applications in
elliptic partial differential equations,
the Morrey space
$M_{r}^{p}\left(\mathbb{R}^{n}\right)$
with $0<r \leq p<\infty$ was introduced
by Morrey \cite{MCB1938} in 1938. From then on, there exists an
increasing interest in applications of
Morrey spaces to various areas of
analysis such as partial differential
equations, potential theory, and harmonic
analysis (see, for instance,
\cite{ADR2015,CFF1987,JHW2009,YSY2010}).

\begin{definition}\label{Def-Morrey}
Let $0<r \leq p<\infty$. The \emph{Morrey space}
$M_{r}^{p}\left(\mathbb{R}^{n}\right)$ is
defined to be the set of all the measurable
functions $f$ on $\rn$ such that
$$
\|f\|_{M_{r}^{p}\left(\mathbb{R}^{n}\right)}:=\sup
_{B \in \mathbb{B}(\rn)}|B|^{\frac{1}{p}-\frac{1}{r}}\|f\|_{L^{r}(B)}<\infty.
$$
\end{definition}

\begin{remark}\label{Rem-Morrey}
Let $0<r \leq p<\infty$. From Definition \ref{Def-Morrey}, we easily deduce
that $M_{r}^{p}(\rn)$ is a ball quasi-Banach function space.
However, it has been pointed out in \cite[p.\,86]{SHYY} that $M_{r}^{p}(\rn)$ may not be a
quasi-Banach function space.
\end{remark}

The following theorem is a corollary of Theorem
\ref{thm-Ia-01}. In what follows,
the \emph{Hardy-Morrey space}
$HM_{r}^{p}(\rn)$ is defined as in Definition
\ref{2d1} with $X:=M_{r}^{p}(\rn)$.

\begin{theorem}\label{Th-Morrey}
Let $\beta\in(1,\fz)$, $\alpha\in(0,n)$, $0<r \leq p<\frac{n}{\alpha}$, and
$I_\alpha$ be the same as in \eqref{cla-I}.
Then $I_\alpha$ can be extended to a bounded linear
operator, still denoted by $I_\alpha$, from
$HM^{p}_r(\rn)$ to $HM^{\beta p}_{\beta
r}(\rn)$, namely, there exists a positive
constant $C$ such that, for any $f\in HM^p_r(\rn)$,
$$\|I_\alpha(f)\|_{HM^{\beta p}_{\beta
r}(\rn)}
\leq C\|f\|_{HM^{p}_{r}(\rn)}
$$
if and only if
$\beta=\frac{n}{n-\alpha p}$.
\end{theorem}

\begin{proof}
Let all the symbols be the same as in the present
theorem. It is easy to show that $M^p_r(\rn)$ is a ball quasi-Banach space.
By \cite[Lemma 2.5]{tx2005} (see also \cite[Lemma 7.2]{zyyw}),
we conclude that
$M_r^p(\rn)$ satisfies Assumption
\ref{assump1} with $p_-:=r$.
Moreover, let $r_0\in(0,\min\{\frac{1}{\beta},r,\frac{n}{n+\alpha}\})$ and
$p_0\in(r_0,\frac{r_0}{1-r_0})$. Then, using
\cite[Lemma 7.6]{zyyw},
we find that Assumption \ref{assump2} with $X:=M_r^p(\rn)$
holds true. Thus, all the assumptions of
Theorem \ref{thm-Ia-01} with $X:=M_r^p(\rn)$ are satisfied.
Besides, it is easy to prove that, for any
$B\in\mathbb{B}(\rn)$,
$$
\|\mathbf{1}_B\|_{M_r^p(\rn)}=|B|^{\frac{1}{p}-\frac{1}{r}}
\|\mathbf{1}_B\|_{L^r(B)}=|B|^{\frac{1}{p}}.
$$
This implies that, for any ball $B\in\mathbb{B}(\rn)$,
\begin{align}\label{MorreyB}
|B|^{\frac{\alpha}{n}}\lesssim\|\mathbf{1}_B\|_
{M^{p}_{r}(\rn)}^{\frac{\beta-1}{\beta}}\
\text{if and only if}\
\beta=\frac{n}{n-\alpha p}.
\end{align}
Then, using Theorem \ref{thm-Ia-01} with $X:=M_r^p(\rn)$, we obtain the desired conclusion,
which completes the proof of
Theorem \ref{Th-Morrey}.
\end{proof}

\begin{remark}
\begin{enumerate}
\item[$\mathrm{(i)}$] Let all the symbols be the same as in Theorem \ref{Th-Morrey}.
By \cite[Lemma 7.2]{zyyw}, Remark \ref{max-re}, and the definition of $HM_r^p(\rn)$, we
find that, if $r\in(1,\infty)$, then $HM_r^p(\rn)=M_r^p(\rn)$.
In this case, Theorem \ref{Th-Morrey} coincides with \cite[Theorem
3.1]{ADR1975} (see also \cite[Corollary 4.7]{GD2020}).
Moreover, to the best of our knowledge,
Theorem \ref{Th-Morrey} is new even when $r\in(0,1]$.
\item[$\mathrm{(ii)}$] To the best of our knowledge, the $p$-convexification of dual spaces of Morrey spaces is unknown with
    $p\in(0,\infty)$, so Theorem \ref{thm-R} can not be applied to Morrey spaces.
\item[$\mathrm{(iii)}$] We point out that there exist many studies
on functional integrals on Morrey-type spaces; we refer the reader to \cite{DGNSS,HSS2016,SHS,SS2017,SST2009} for this.
\end{enumerate}
\end{remark}

Using Theorem \ref{coclassic}, we can obtain the following conclusion.

\begin{theorem}\label{co-Morrey}
Let $\beta\in(1,\fz)$, $\alpha\in(0,n)$, $1<r \leq p<\frac{n}{\alpha}$, and
$\mathcal{M}_\alpha$ be the same as in \eqref{de-Ma}.
Then $\mathcal{M}_\alpha$ is bounded from
$M^{p}_r(\rn)$ to $M^{\beta p}_{\beta
r}(\rn)$, namely, there exists a positive
constant $C$ such that, for any $f\in M^p_r(\rn)$,
$$\|\mathcal{M}_\alpha(f)\|_{M^{\beta p}_{\beta
r}(\rn)}
\leq C\|f\|_{M^{p}_{r}(\rn)}
$$
if and only if
$\beta=\frac{n}{n-\alpha p}$.
\end{theorem}

\begin{proof}
Let all the symbols be the same as in the present theorem.
By \cite[Lemma 2.5]{tx2005} (see also \cite[Lemma 7.2]{zyyw}),
we conclude that
$M_r^p(\rn)$ satisfies Assumption
\ref{assump1} with $p_-:=r\in(1,\infty)$.
From this and Remark \ref{max-re}, we infer that $M^{p}_{r}(\rn)$ satisfies
Assumption \ref{max-assump}.
Thus, all the assumptions of Theorem \ref{coclassic}
with $X:=M^{p}_{r}(\rn)$
are satisfied.
Then, using \eqref{MorreyB} and Theorem \ref{coclassic} with $X:=M^{p}_{r}(\rn)$, we obtain the desired conclusion, which completes the proof of Theorem \ref{co-Morrey}.
\end{proof}

\subsection{Mixed-Norm Lebesgue Spaces\label{Mixed-Norm}}

The mixed-norm Lebesgue space
$L^{\vec{p}}\left(\mathbb{R}^{n}\right)$
was studied by Benedek and Panzone
\cite{BAP1961} in 1961, which can be
traced back to H\"ormander \cite{HL1960}.
Later on, in 1970, Lizorkin \cite{LPI1970}
further developed both the theory of
multipliers of Fourier integrals and
estimates of convolutions in the
mixed-norm Lebesgue spaces. Particularly,
in order to meet the requirements arising
in the study of the boundedness of
operators, partial differential equations,
and some other analysis subjects, the
real-variable theory of mixed-norm
function spaces, including mixed-norm
Morrey spaces, mixed-norm Hardy spaces,
mixed-norm Besov spaces, and mixed-norm
Triebel-Lizorkin spaces, has rapidly been
developed in recent years (see, for
instance,
\cite{CGN2017,CGG2017,CGN20172,GN2016,HLY2019,HLYY2019,HYacc,NT2019}).

\begin{definition}\label{mixed}
Let $\vec{p}:=(p_{1}, \ldots, p_{n})
\in(0, \infty]^{n}$. The mixed-norm
Lebesgue space
$L^{\vec{p}}\left(\mathbb{R}^{n}\right)$
is defined to be the set of all the measurable
functions $f$ on $\rn$ such that
\begin{align*}
\|f\|_{L^{\vec{p}}\left(\mathbb{R}^{n}\right)}:=\left\{\int_{\mathbb{R}}
\cdots\left[\int_{\mathbb{R}}\left|f\left(x_{1}, \ldots,
x_{n}\right)\right|^{p_{1}} \,d x_{1}\right]^{\frac{p_{2}}{p_{1}}}
\cdots \,dx_{n}\right\}^{\frac{1}{p_{n}}}<\infty
\end{align*}
with the usual modifications made when $p_i=\infty$ for some $i\in\{1,\ldots,n\}$.
\end{definition}

\begin{remark}\label{mix-r}
Let $\vec{p}\in(0, \infty)^{n}$.	
By Definition \ref{mixed}, we easily conclude that
$L^{\vec{p}}\left(\mathbb{R}^{n}\right)$
is a ball quasi-Banach space,
but, it is worth pointing out that $L^{\vec{p}}(\rn)$
may not be a quasi-Banach function space
(see, for instance, \cite[Remark 7.20]{zyyw}).
\end{remark}

The following theorem is a corollary of
Theorem \ref{thm-Ia-02}. In what follows,
the \emph{mixed-norm Hardy space}
$H^{\vec{p}}(\rn)$ is defined as in
Definition \ref{2d1} with
$X:=L^{\vec{p}}(\rn)$.

\begin{theorem}\label{Th-Mixed}
Let $\beta\in(1,\infty)$, $\alpha\in(0,n)$, $\vec{p}:=(p_{1}, \ldots, p_{n})
\in(0, \infty)^{n}$ satisfy
$\sum_{i=1}^{n}\frac{1}{p_i}\in(\alpha,\infty)$,
and $I_\alpha$ be the same as in \eqref{cla-I}.
Then $I_\alpha$ can be extended to a unique bounded linear
operator, still denoted by $I_\alpha$, from
$H^{\vec{p}}(\rn)$ to
$H^{\beta\vec{p}}(\rn)$, namely, there exists a positive
constant $C$ such that, for any $f\in H^{\vec{p}}(\rn)$,
$$\|I_\alpha(f)\|_{H^{\beta\vec{p}}(\rn)}
\leq C\|f\|_{H^{\vec{p}}(\rn)}
$$
if and only if
$\beta=\frac{\sum_{i=1}^{n}\frac{1}{p_i}}
{\sum_{i=1}^{n}\frac{1}{p_i}-\alpha}$.
\end{theorem}

\begin{proof}
Let all the symbols be the same as in the
present theorem. It is easy to show that
$L^{\vec{p}}\left(\mathbb{R}^{n}\right)$
is a ball quasi-Banach space with an absolutely continuous quasi-norm.
Let
\begin{align}\label{p-p+}
p_{-}:=\min\{p_{1},
\ldots, p_{n}\}\text{ and } p_{+}:=\max\{p_{1},
\ldots, p_{n}\}.
\end{align}
Then, by \cite[Lemma
3.5]{HLY2019}, we conclude that
$L^{\vec{p}}(\rn)$ satisfies Assumption
\ref{assump1} with $p_-$ the same as in \eqref{p-p+}.
Moreover, letting $r_0\in(0,\min\{\frac{1}{\beta},p_-\})$ and
$p_0\in(p_+,\infty)$, by both
the dual theorem of
$L^{\vec{p}}(\rn)$ (see \cite[p.\,304,
Theorem 1.a]{BAP1961}) and
\cite[Lemma 3.5]{HLY2019},
we conclude that Assumption \ref{assump2}
with $X:=L^{\vec{p}}(\rn)$
also holds true. Thus, all the assumptions of Theorem \ref{thm-Ia-02} with
$X:=L^{\vec{p}}(\rn)$ are satisfied.

Besides, using the definition of the
mixed-norm Lebesgue space, we have, for any
$B:=B(x,r)\in\mathbb{B}(\rn)$ with $x\in\rn$
and $r\in(0,\infty)$,
\begin{align*}
\left\|\mathbf{1}_{B}\right\|_{L^{\vec{p}}(\mathbb{R}^{n})}
&=\left\|\mathbf{1}_{B(\mathbf{0},r)}\right\|_{L^{\vec{p}}(\mathbb{R}^{n})}
=\left\{\int_{-r}^{r}
\cdots\left[\int_{-\sqrt{
r^{2}-\left(x_{2}^{2}+\cdots
x_{n}^{2}\right)}}^{\sqrt{
r^{2}-\left(x_{2}^{2}+\cdots
x_{n}^{2}\right)}}
\,dx_{1}\right]^{\frac{p_{2}}{p_{1}}}
\cdots
\,dx_{n}\right\}^{\frac{1}{p_{n}}}\\
&=r^{\sum_{i=1}^{n}\frac{1}{p_i}}
\left\{\int_{-1}^{1}
\cdots\left[\int_{-\sqrt{1-\left(x_{2}^{2}+\cdots
x_{n}^{2}\right)}}^{\sqrt{1-\left(x_{2}^{2}+\cdots
x_{n}^{2}\right)}}
\,dx_{1}\right]^{\frac{p_{2}}{p_{1}}} \cdots
\,dx_{n}\right\}^{\frac{1}{p_{n}}}\\
&=r^{\sum_{i=1}^{n}\frac{1}{p_i}}
\left\|\mathbf{1}_{B(\mathbf{0},1)}\right\|
_{L^{\vec{p}}(\mathbb{R}^{n})}.
\end{align*}
This implies that, for any ball $B\in\mathbb{B}(\rn)$,
\begin{align}\label{MixB}
|B|^{\frac{\alpha}{n}}\lesssim\|\mathbf{1}_{B}\|
_{L^{\vec{p}}(\mathbb{R}^{n})}^{\frac{\beta-1}{\beta}}
\ \text{if and only if}\  \beta=\frac{\sum_{i=1}^{n}\frac{1}{p_i}}
{\sum_{i=1}^{n}\frac{1}{p_i}-\alpha}.
\end{align}
Then, using Theorem
\ref{thm-Ia-02} with
$X:=L^{\vec{p}}(\rn)$, we obtain the desired conclusion, which completes the proof
of Theorem \ref{Th-Mixed}.
\end{proof}

The following theorem is a corollary of
Theorem \ref{thm-R}. In what follows, let
$\vec{p}':=(p_1',\ldots,p_n')$ for any
$\vec{p}=(p_1,\ldots,p_n)\in[1,\infty]^n$,
where $1/p_i+1/p_i'=1$ for any $i\in\{1,\ldots,n\}$.

\begin{theorem}\label{lyq-X-Y}
Let $\alpha\in(0,n)$, $I_\alpha$ be the same as in \eqref{cla-I},
$\vec{p}:=(p_{1}, \ldots, p_{n})
\in(0, \infty)^{n}$ satisfy
$\sum_{i=1}^{n}\frac{1}{p_i}\in(\alpha,\infty)$, and $\vec{q}:=(\frac{np_1}{n-\alpha p_1},
\ldots, \frac{np_n}{n-\alpha p_n})$.
Then $I_\alpha$ can be extended to a unique bounded linear
operator, still denoted by $I_\alpha$, from
$H^{\vec{p}}(\rn)$ to
$H^{\vec{q}}(\rn)$, namely, there exists a positive
constant $C$ such that, for any $f\in H^{\vec{p}}(\rn)$,
$$\|I_\alpha(f)\|_{H^{\vec{q}}(\rn)}
\leq C\|f\|_{H^{\vec{p}}(\rn)}.
$$
\end{theorem}

\begin{proof}
Let all the symbols be the same as in the present theorem.
Let $p_0\in(0,\min\{p_-,\frac{n}{n+\alpha}\})$ with $p_-$
being the same as in \eqref{p-p+}, and
$\frac{1}{q_0}:=\frac{1}{p_0}-\frac{\alpha}{n}$.
It is easy to show that both $L^{\frac{\vec{p}}{p_0}}(\rn)$ and $L^{\frac{\vec{q}}{q_0}}(\rn)$
are ball Banach function spaces, and
$$\lf(\lf[L^{\vec{q}}(\rn)\r]^{\frac{1}{q_0}}\r)'
=L^{(\frac{\vec{q}}{q_0})'}(\rn)
=L^{\frac{p_0}{q_0}(\frac{\vec{p}}{p_0})'}(\rn)
=\lf[\lf(\lf[L^{\vec{p}}(\rn)\r]^{\frac{1}{p_0}}\r)'\r]^{\frac{p_0}{q_0}}.$$
By \cite[Lemma 3.5]{HLY2019}, we conclude that the Hardy--Littlewood maximal
operator $\mathcal{M}$ is bounded on $L^{(\frac{\vec{q}}{q_0})'}(\rn)$.
Thus, all the assumptions of Theorem
\ref{thm-R} with both $X:=L^{\vec{p}}(\rn)$ and
$Y:=L^{\vec{q}}(\rn)$ are satisfied.
Then, using Theorem \ref{thm-R} with both $X:=L^{\vec{p}}(\rn)$ and
$Y:=L^{\vec{q}}(\rn)$, we obtain the desired conclusion, which completes
the proof of Theorem \ref{lyq-X-Y}.
\end{proof}

\begin{remark}
Let all the symbols be the same as in Theorem \ref{Th-Mixed}.
Assume that $p_-\in(1,\infty)$ is the same as in \eqref{p-p+}.
Then, by \cite[Lemma 3.5]{HLY2019}, Remark \ref{max-re}, and the definition of $H^{\vec{p}}(\rn)$,
we have $H^{\vec{p}}(\rn)=L^{\vec{p}}(\rn)$.
In this case, in \cite[Lemma 3.1]{ZZ2022}, Zhang and Zhou proved that $I_\alpha$ is bounded from $L^{\vec{p}}(\rn)$ to $L^{\vec{q}}(\rn)$
if and only if
\begin{align}\label{cjy}
\alpha=\sum_{i=1}^{n}\frac{1}{p_i}
-\sum_{i=1}^{n}\frac{1}{q_i},
\end{align}
where
$\vec{p}:=(p_{1}, \ldots, p_{n}),\vec{q}:=(q_{1}, \ldots, q_{n})\in[1,\infty)^n$ satisfy
$p_i\leq q_i$, for any $i\in\{1,\ldots,n\}$, and $p_n\in(1,\infty)$.

When $n=1$ and $p_1\in(1,\frac{1}{\alpha})$,
it is easy to show that both \cite[Lemma 3.1]{ZZ2022} and Theorem
\ref{Th-Mixed} in this case coincide with Lemma \ref{fractional}(ii).

Let $n\geq2$ and $\vec{p}:=(p_1,p_2)\in[1,\infty)^2$ and
$\frac{1}{p_1}+\frac{1}{p_2}\in(\alpha,\infty)$.
Now, we claim that \cite[Lemma 3.1]{ZZ2022} in this case gives a more general
conclusion than both Theorems \ref{Th-Mixed} and \ref{lyq-X-Y} in this case.
Let $\beta:=\frac{\sum_{i=1}^{n}\frac{1}{p_i}}
{\sum_{i=1}^{n}\frac{1}{p_i}-\alpha}$,
$\vec{q}_1:=\beta\vec{p}$, and $\vec{q}_2$ be the same as in Theorem \ref{lyq-X-Y}.
Then \eqref{cjy} holds true with $\vec{q}$ replaced by $\vec{q}_1$ or $\vec{q}_2$.
Using \cite[Lemma 3.1]{ZZ2022}, we conclude that
$I_\alpha$ is bounded from $L^{\vec{p}}(\rn)$ to $L^{\vec{q_1}}(\rn)$
and also from $L^{\vec{p}}(\rn)$ to $L^{\vec{q_2}}(\rn)$.
Moreover, \cite[Lemma 3.1]{ZZ2022} can be applied to the case that $\vec{p}=(3,\frac{3}{2})$ and
$\vec{q}=(3,6)$, but Theorems \ref{Th-Mixed} and \ref{lyq-X-Y} can not.
Thus, \cite[Lemma 3.1]{ZZ2022} gives a more general conclusion than Theorems \ref{Th-Mixed} and \ref{lyq-X-Y}.

However, to the best of our knowledge,
Theorems \ref{Th-Mixed} and \ref{lyq-X-Y} are new when $p_-\in(0,1]$.
\end{remark}

Using Theorem \ref{coclassic}, we can obtain the following conclusion.

\begin{theorem}\label{co-Mixed}
Let $\beta\in(1,\infty)$, $\alpha\in(0,n)$, $\vec{p}:=(p_{1}, \ldots, p_{n})
\in(0, \infty)^{n}$ satisfy
$\sum_{i=1}^{n}\frac{1}{p_i}\in(\alpha,\infty)$,
$p_-\in(1,\infty)$ be the same as in \eqref{p-p+},
and $\mathcal{M}_\alpha$ the same as in \eqref{de-Ma}.
Then $\mathcal{M}_\alpha$ is bounded from
$L^{\vec{p}}(\rn)$ to
$L^{\beta\vec{p}}(\rn)$, namely, there exists a positive
constant $C$ such that, for any $f\in L^{\vec{p}}(\rn)$,
$$\|\mathcal{M}_\alpha(f)\|_{L^{\beta\vec{p}}(\rn)}
\leq C\|f\|_{L^{\vec{p}}(\rn)}
$$
if and only if
$\beta=\frac{\sum_{i=1}^{n}\frac{1}{p_i}}
{\sum_{i=1}^{n}\frac{1}{p_i}-\alpha}$.
\end{theorem}

\begin{proof}
Let all the symbols be the same as in the present theorem.
From both $p_-\in(1,\infty)$ and Remark \ref{max-re}, we infer that
$L^{\vec{p}}(\rn)$ satisfies Assumption \ref{max-assump}.
Thus, all the assumptions of Theorem \ref{coclassic}
with $X:=L^{\vec{p}}(\rn)$
are satisfied.
Then, using \eqref{MixB} and Theorem \ref{coclassic} with $X:=L^{\vec{p}}(\rn)$,
we obtain the desired conclusion, which completes the proof of Theorem \ref{co-Mixed}.
\end{proof}

\subsection{Local Generalized Herz Spaces\label{H-H}}

The local generalized Herz space was originally introduced recently
by Rafeiro and Samko \cite{RS2020},
which is a generalization of the classical
homogeneous Herz space and connects with the generalized Morrey type space.
Later on, Li et al. \cite{LYH2022} studied
the real-variable theory of Hardy spaces
associated with the local generalized Herz space.
Now, we recall the definitions of both the function class
$M\left(\mathbb{R}_{+}\right)$ and the
local generalized Herz space
$\dot{\mathcal{K}}_{\omega, \mathbf{0}}^{p,
q}(\mathbb{R}^{n})$
(see \cite[Definitions 2.1 and 2.2]{RS2020}
and also \cite[Definitions 1.1.1 and 1.2.1]{LYH2022}).

\begin{definition}
Let $\mathbb{R}_+:=(0,\infty)$. The \emph{function class}
$M\left(\mathbb{R}_{+}\right)$ is defined to
be the set of all the positive functions
$\omega$ on $\mathbb{R}_{+}$ such that, for
any $0<\delta<N<\infty$,
$$
0<\inf _{t \in(\delta, N)} \omega(t) \leq
\sup _{t \in(\delta, N)} \omega(t)<\infty
$$
and there exist four constants $\alpha_{0}$,
$\beta_{0}$, $\alpha_{\infty}$,
$\beta_{\infty} \in \mathbb{R}$ such that
\begin{itemize}
\item[\rm (i)] for any $t \in(0,1]$, $\omega(t)
t^{-\alpha_{0}}$ is almost increasing and
$\omega(t) t^{-\beta_{0}}$ is almost
decreasing;

\item[\rm (ii)] for any $t \in[1, \infty)$, $\omega(t)
t^{-\alpha_{\infty}}$ is almost increasing
and $\omega(t) t^{-\beta_{\infty}}$ is
almost decreasing.
\end{itemize}
\end{definition}

\begin{definition}
Let $p,q \in(0, \infty)$ and $\omega \in
M\left(\mathbb{R}_{+}\right)$.
The \emph{local generalized Herz space}
$\dot{\mathcal{K}}_{\omega, \mathbf{0}}^{p,
q}(\mathbb{R}^{n})$ is defined to
be the set of all the measurable functions
$f$ on $\mathbb{R}^{n}$ such that
$$
\|f\|_{\dot{\mathcal{K}}_{\omega, \mathbf{0}}^{p,
q}\left(\mathbb{R}^{n}\right)}:=\left\{\sum_{k \in
\mathbb{Z}}\left[\omega\left(2^{k}\right)\right]^{q}\left\|f
\mathbf{1}_{B\left(\mathbf{0}, 2^{k}\right) \setminus
B\left(\mathbf{0},
2^{k-1}\right)}\right\|_{L^{p}\left(\mathbb{R}^{n}\right)}
^{q}\right\}^{\frac{1}{q}}
$$
is finite.
\end{definition}

\begin{definition}
Let $\omega$ be a positive
function on $\mathbb{R}_{+}$. Then the
\emph{Matuszewska-Orlicz indices} $m_{0}(\omega)$,
$M_{0}(\omega)$, $m_{\infty}(\omega)$, and
$M_{\infty}(\omega)$ of $\omega$ are
defined, respectively, by setting, for any
$h \in(0, \infty)$,
$$m_{0}(\omega):=\sup _{t \in(0,1)}
\frac{\ln (\varlimsup\limits_{h \to
0^{+}} \frac{\omega(h
t)}{\omega(h)})}{\ln t},$$
$$M_{0}(\omega):=\inf _{t \in(0,1)}
\frac{\ln (
\varliminf\limits_{h\to0^{+}}
\frac{\omega(h t)}{\omega(h)})}{\ln
t},$$
$$m_{\infty}(\omega):=\sup _{t \in(1,
\infty)} \frac{\ln (\varliminf
\limits_{h \to\infty}
\frac{\omega(h t)}{\omega(h)})}{\ln
t},$$
and
$$
M_{\infty}(\omega):=\inf _{t \in(1, \infty)}
\frac{\ln (\varlimsup\limits_{h
\to \infty} \frac{\omega(h
t)}{\omega(h)})}{\ln t}.
$$
\end{definition}

\begin{remark}
From \cite[Theorem 1.2.20]{LYH2022},
we infer that $\dot{\mathcal{K}}_{\omega, \mathbf{0}}^{p,q}(\mathbb{R}^{n})$ is a
ball quasi-Banach function space with $p,q\in(0,\infty)$ and $\omega\in M(\mathbb{R}_+)$ satisfying
$m_0(\omega)\in(-\frac{n}{p},\infty)$.
However, it is worth pointing out that
$\dot{\mathcal{K}}_{\omega, \mathbf{0}}^{p,q}(\mathbb{R}^{n})$ with
$p,q \in(0, \infty)$ and $\omega \in
M\left(\mathbb{R}_{+}\right)$ may not be a quasi-Banach function space.
For instance, let $p:=1=:q$ and $\omega(t):=t^n$ for any $t\in(0,\infty)$.
In this case, $\dot{\mathcal{K}}_{\omega, \mathbf{0}}^{p,q}(\mathbb{R}^{n})
=\dot{\mathcal{K}}_{\omega, \mathbf{0}}^{1,1}(\mathbb{R}^{n})$.
Let $\vec{e}:=(1,0,\ldots,0)$ and
$$E:=\bigcup_{k\in\nn}B(3\times2^k\vec{e},2^{-k}).
$$
Then it is easy to show that $|E|<\infty$, but
$$\lf\|\mathbf{1}_E\r\|_{\dot{\mathcal{K}}_{\omega, \mathbf{0}}^{1,1}(\mathbb{R}^{n})}
=\sum_{k\in\nn}2^{nk}\lf|B\lf(3\times2^k\vec{e},2^{-k}\r)\r|
=\sum_{k\in\nn}|B(\mathbf{0},1)|=\infty.
$$
Thus, $\dot{\mathcal{K}}_{\omega, \mathbf{0}}^{1,1}(\mathbb{R}^{n})$
is not a quasi-Banach function space.
\end{remark}

The following theorem is a corollary of Theorem \ref{thm-Ia-02}. The
\emph{generalized Herz Hardy
space} $H\dot{\mathcal{K}}_{\omega,
\mathbf{0}}^{p,q}(\rn)$ is defined as in
Definition \ref{2d1} with
$X:=\dot{\mathcal{K}}_{\omega,\mathbf{0}}^{p,
q}(\mathbb{R}^{n})$ (see \cite[Definition 4.0.15]{LYH2022}).

\begin{theorem}\label{Th-LGH1}
Let $p,q\in(0,\infty)$ and $\omega\in M(\mathbb{R}_+)$ satisfy
$m_0(\omega)\in(-\frac{n}{p},\infty)$
and $m_\infty(\omega)\in(-\frac{n}{p},\infty)$.
Let $\beta\in(1,\fz)$, $\alpha\in(0,n)$,
and $I_\alpha$ be the same as in \eqref{cla-I}.
Then $I_\alpha$ can be extended to a unique bounded linear
operator, still denoted by $I_\alpha$, from $H\dot{\mathcal{K}}_{\omega,
\mathbf{0}}^{p,q}(\rn)$ to
$H\dot{\mathcal{K}}_{\omega^{1/\beta},
\mathbf{0}}^{\beta p,\beta
q}(\rn)$, namely,
there exists a positive constant $C$ such that, for any $f\in H\dot{\mathcal{K}}_{\omega,
\mathbf{0}}^{p,q}(\rn)$,
$$
\|I_\alpha(f)\|_{H\dot{\mathcal{K}}_{\omega^{1/\beta},
\mathbf{0}}^{\beta p,\beta
q}(\rn)}\leq C\|f\|_{H\dot{\mathcal{K}}_{\omega,
\mathbf{0}}^{p,q}(\rn)}
$$
if and only if there exists a positive constant
$\wz C$ such that, for any ball
$B\in\mathbb{B}(\rn)$,
\begin{align*}
|B|^{\frac{\alpha}{n}}\leq \wz C
\|\mathbf{1}_B\|_{\dot{\mathcal{K}}_{\omega,\mathbf{0}}^{p,
q}(\mathbb{R}^{n})}^{\frac{\beta-1}{\beta}}.
\end{align*}
\end{theorem}
\begin{proof}
Let all the symbols be the same as in
the present theorem. By \cite[Theorems 1.2.20 and 1.4.1]{LYH2022},
we conclude that
$\dot{\mathcal{K}}_{\omega,\mathbf{0}}^{p,
q}(\mathbb{R}^{n})$ is a ball quasi-Banach space with an absolutely continuous quasi-norm.
To prove the desired conclusion,
we claim that all the assumptions
of Theorem \ref{thm-Ia-02} with
$X:=\dot{\mathcal{K}}_{\omega,
\mathbf{0}}^{p,q}(\rn)$ are satisfied. Indeed,
$\dot{\mathcal{K}}_{\omega,
\mathbf{0}}^{p,q}(\rn)$ satisfies Assumption
\ref{assump1} with
\begin{align}\label{p-lyq}
p_-:=\min \left\{p, \frac{n}{\max
\left\{M_{0}(\omega),
M_{\infty}(\omega)\right\}+n /
p}\right\};
\end{align}
see \cite[Lemma 4.3.8]{LYH2022}.
Moreover, from \cite[Lemma 1.8.5]{LYH2022}, we deduce
that Assumption \ref{assump2} with $X:=\dot{\mathcal{K}}_{\omega,
\mathbf{0}}^{p,q}(\rn)$
holds true for any given
$$
r_0 \in\left(0, \min \left\{\frac{1}{\beta},p_-, q\right\}\right)
$$
and
$$
p_0 \in\left(\max \left\{p, \frac{n}{\min \left\{m_{0}(\omega),
m_{\infty}(\omega)\right\}+n / p}\right\}, \infty\right].
$$
Thus, the above claim holds true.
Then, using Theorem \ref{thm-Ia-02} with $X:=\dot{\mathcal{K}}_{\omega,
\mathbf{0}}^{p,q}(\rn)$, we obtain the desired conclusion, which
completes the proof of Theorem \ref{Th-LGH1}.
\end{proof}

The following theorem is a corollary of Theorem \ref{Th-LGH1}.

\begin{theorem}\label{Th-LGH}
Let all the symbols be the same as in Theorem \ref{Th-LGH1}.
Further assume that $m_\infty(\omega)\geq M_0(\omega)>0$.
Then, when
$M_0(\omega)\leq\alpha-\frac{n}{p}$ and
$\beta\in[\frac{n+m_\infty(\omega)p}{n+m_\infty(\omega)p-\alpha
p},\infty)$,
or when
$M_0(\omega)>\alpha-\frac{n}{p}$ and
$\beta\in[\frac{n+m_\infty(\omega)p}{n+m_\infty(\omega)p-\alpha p},
\frac{n+M_0(\omega)p}{n+M_0(\omega)p-\alpha p}]$,
$I_\alpha$ can be extended to a unique bounded linear
operator, still denoted by $I_\alpha$, from $H\dot{\mathcal{K}}_{\omega,
\mathbf{0}}^{p,q}(\rn)$ to
$H\dot{\mathcal{K}}_{\omega^{1/\beta},
\mathbf{0}}^{\beta p,\beta
q}(\rn)$, namely,
there exists a positive constant $C$ such that, for any $f\in H\dot{\mathcal{K}}_{\omega,
\mathbf{0}}^{p,q}(\rn)$,
$$
\|I_\alpha(f)\|_{H\dot{\mathcal{K}}_{\omega^{1/\beta},
\mathbf{0}}^{\beta p,\beta
q}(\rn)}\leq C\|f\|_{H\dot{\mathcal{K}}_{\omega,
\mathbf{0}}^{p,q}(\rn)}.
$$
\end{theorem}

\begin{proof}
Let all the symbols be the same as in
the present theorem.
To prove the present theorem,
using Theorem \ref{Th-LGH1},
it suffices to show that,
for any ball $B\in\mathbb{B}(\rn)$,
\begin{align}\label{Herz-ineq}
|B|^{\frac{\alpha}{n}}\leq C
\|\mathbf{1}_B\|_{\dot{\mathcal{K}}_{\omega,\mathbf{0}}^{p,
q}(\mathbb{R}^{n})}^{\frac{\beta-1}{\beta}}.
\end{align}
Indeed, from
\cite[(4.9.12) and (4.9.13)]{LYH2022}, we deduce that, for any $k\in[M_0(\omega),m_\infty(\omega)]$
and any ball $B:=B(x_0,r)\in\mathbb{B}(\rn)$ with $x_0\in\rn$
and $r\in(0,\infty)$,
\begin{align}\label{lyqB}
\lf\|\mathbf{1}_B\r\|_{\dot{\mathcal{K}}_{\omega,
\mathbf{0}}^{p,q}(\rn)}\gtrsim
r^{\frac{n}{p}}\omega(r)\gtrsim
r^{\frac{n}{p}+k},
\end{align}
which implies that \eqref{Herz-ineq} holds true. This finishes the proof of
Theorem \ref{Th-LGH}.
\end{proof}

Using Theorem \ref{coclassic}, we can obtain the following conclusion.

\begin{theorem}\label{co-lyq}
Let all the symbols be the same as in Theorem \ref{Th-LGH}, $p_-\in(1,\infty)$ the same as in \eqref{p-lyq}, and
$\mathcal{M}_\alpha$ the same as in \eqref{de-Ma}.
Then, when
$M_0(\omega)\leq\alpha-\frac{n}{p}$ and
$\beta\in[\frac{n+m_\infty(\omega)p}{n+m_\infty(\omega)p-\alpha
p},\infty)$,
or when
$M_0(\omega)>\alpha-\frac{n}{p}$ and
$\beta\in[\frac{n+m_\infty(\omega)p}{n+m_\infty(\omega)p-\alpha p},
\frac{n+M_0(\omega)p}{n+M_0(\omega)p-\alpha p}]$,
$\mathcal{M}_\alpha$ is bounded from $\dot{\mathcal{K}}_{\omega,
\mathbf{0}}^{p,q}(\rn)$ to
$\dot{\mathcal{K}}_{\omega^{1/\beta},
\mathbf{0}}^{\beta p,\beta
q}(\rn)$, namely,
there exists a positive constant $C$ such that, for any $f\in \dot{\mathcal{K}}_{\omega,
\mathbf{0}}^{p,q}(\rn)$,
$$
\|\mathcal{M}_\alpha(f)\|_{\dot{\mathcal{K}}_{\omega^{1/\beta},
\mathbf{0}}^{\beta p,\beta
q}(\rn)}\leq C\|f\|_{\dot{\mathcal{K}}_{\omega,
\mathbf{0}}^{p,q}(\rn)}.
$$
\end{theorem}

\begin{proof}
Let all the symbols be the same as in the present theorem.
From $p_-\in(1,\infty)$ and Remark \ref{max-re}, we infer that
$\dot{\mathcal{K}}_{\omega,
\mathbf{0}}^{p,q}(\rn)$ satisfies Assumption \ref{max-assump}.
Thus, all the assumptions of Theorem \ref{coclassic}
with $X:=\dot{\mathcal{K}}_{\omega,
\mathbf{0}}^{p,q}(\rn)$
are satisfied.
Then, using \eqref{lyqB} and Theorem \ref{coclassic} with $X:=\dot{\mathcal{K}}_{\omega,
\mathbf{0}}^{p,q}(\rn)$, we
obtain the desired conclusion, which completes the proof of Theorem \ref{co-lyq}.
\end{proof}

The following theorem is a corollary of Theorem \ref{thm-R}.

\begin{theorem}\label{lyq-X-Y-Z}
Let $\alpha\in(0,n)$ and $I_\alpha$ be the same as in \eqref{cla-I}.	
Let $p,q\in(0,\frac{n}{\alpha})$ and $\omega\in M(\mathbb{R}_+)$ satisfy
$m_0(\omega),m_{\fz}(\rn)\in(\alpha-\frac{n}{p},\infty)$.
Then $I_\alpha$ can be extended to a unique bounded linear
operator, still denoted by $I_\alpha$, from
$H\dot{\mathcal{K}}_{\omega,\mathbf{0}}^{p,q}(\rn)$ to
$H\dot{\mathcal{K}}_{\omega,\mathbf{0}}^{\frac{np}{n-\alpha p},\frac{nq}{n-\alpha q}}(\rn)$, namely,
there exists a positive constant $C$ such that, for any $f\in H\dot{\mathcal{K}}_{\omega,
\mathbf{0}}^{p,q}(\rn)$,
$$
\|I_\alpha(f)\|_{H\dot{\mathcal{K}}_{\omega,\mathbf{0}}^
{\frac{np}{n-\alpha p},\frac{nq}{n-\alpha p}}(\rn)}
\leq C\|f\|_{H\dot{\mathcal{K}}_{\omega,
\mathbf{0}}^{p,q}(\rn)}.
$$
\end{theorem}

\begin{proof}
Let all the symbols be the same as in the present theorem.
From \cite[Theorems 1.2.20 and 1.4.1]{LYH2022},
we deduce that both $\dot{\mathcal{K}}_{\omega,\mathbf{0}}^{p,q}(\mathbb{R}^{n})$
and $\dot{\mathcal{K}}_{\omega,\mathbf{0}}^{\frac{np}{n-\alpha p},\frac{nq}{n-\alpha q}}(\rn)$
are ball quasi-Banach spaces with absolutely continuous quasi-norms.
Let
\begin{align}\label{lyq-X-Y-01}
p_0\in\lf(0,\min\lf\{p,q,\frac{n}{n+\alpha},\frac{n}{\max
\left\{M_{0}(\omega),M_{\infty}(\omega)\right\}+n/p}\r\}\r)
\end{align}
and
\begin{align}\label{lyq-X-Y-02}
\frac{1}{q_0}:=\frac{1}{p_0}-\frac{\alpha}{n}.
\end{align}
By \eqref{lyq-X-Y-01} and \eqref{lyq-X-Y-02}, we have
\begin{align}\label{q014}
q_0\in\lf(0,\min\lf\{\frac{np}{n-\alpha p},\frac{nq}{n-\alpha q},
\frac{n}{\max\left\{M_{0}(\omega),M_{\infty}(\omega)\right\}+(n-\alpha p)/p}\r\}\r)
\cap(0,1],
\end{align}
which, together with \eqref{lyq-X-Y-01} and \cite[Lemma 1.8.5]{LYH2022}, further implies that
both $[\dot{\mathcal{K}}_{\omega,\mathbf{0}}^{p,q}(\rn)]^{\frac{1}{p_0}}$
and $[\dot{\mathcal{K}}_{\omega,\mathbf{0}}^
{\frac{np}{n-\alpha p},\frac{nq}{n-\alpha q}}(\rn)]^{\frac{1}{q_0}}$
are ball Banach function spaces.

Next, we show
\begin{align}\label{lyq-X-Y-07}
\lf(\lf[\dot{\mathcal{K}}_{\omega,\mathbf{0}}^
{\frac{np}{n-\alpha p},\frac{nq}{n-\alpha q}}(\rn)\r]^{\frac{1}{q_0}}\r)'
=\lf[\lf(\lf[\dot{\mathcal{K}}_{\omega,\mathbf{0}}^{p,q}(\rn)\r]
^{\frac{1}{p_0}}\r)'\r]^{\frac{p_0}{q_0}}.
\end{align}
Indeed, by \cite[Lemma 1.3.1 and Theorem 1.7.6]{LYH2022}, we have
\begin{align}\label{lyq-X-Y-03}
\lf(\lf[\dot{\mathcal{K}}_{\omega,\mathbf{0}}^
{\frac{np}{n-\alpha p},\frac{nq}{n-\alpha q}}(\rn)\r]^{\frac{1}{q_0}}\r)'
&=\lf(\dot{\mathcal{K}}_{\omega^{q_0},\mathbf{0}}^
{\frac{np}{[n-\alpha p]q_0},\frac{nq}{[n-\alpha q]q_0}}(\rn)\r)'\\
&=\dot{\mathcal{K}}_{\omega^{-q_0},\mathbf{0}}^
{(\frac{np}{[n-\alpha p]q_0})',(\frac{nq}{[n-\alpha q]q_0})'}(\rn)\noz
\end{align}
and
\begin{align}\label{lyq-X-Y-04}
\lf[\lf(\lf[\dot{\mathcal{K}}_{\omega,\mathbf{0}}^{p,q}(\rn)\r]
^{\frac{1}{p_0}}\r)'\r]^{\frac{p_0}{q_0}}
&=\lf(\lf[\dot{\mathcal{K}}_{\omega^{p_0},\mathbf{0}}
^{\frac{p}{p_0},\frac{q}{p_0}}(\rn)\r]'\r)^{\frac{p_0}{q_0}}\\
&=\lf[\dot{\mathcal{K}}_{\omega^{-p_0},\mathbf{0}}
^{(\frac{p}{p_0})',(\frac{q}{p_0})'}(\rn)\r]^{\frac{p_0}{q_0}}
=\dot{\mathcal{K}}_{\omega^{-q_0},\mathbf{0}}
^{(\frac{p}{p_0})'(\frac{p_0}{q_0}),(\frac{q}{p_0})'\frac{p_0}{q_0}}(\rn).\noz
\end{align}
Moreover, using \eqref{lyq-X-Y-02}, we conclude that
$(\frac{np}{[n-\alpha p]q_0})'=(\frac{p}{p_0})'(\frac{p_0}{q_0})$
and $(\frac{nq}{[n-\alpha q]q_0})'=(\frac{q}{p_0})'\frac{p_0}{q_0}$,
which, combined with \eqref{lyq-X-Y-03} and \eqref{lyq-X-Y-04}, further implies that
\eqref{lyq-X-Y-07} holds true.

We now show that $\mathcal{M}$ is bounded on $([\dot{\mathcal{K}}_{\omega,\mathbf{0}}^
{\frac{np}{n-\alpha p},\frac{nq}{n-\alpha q}}(\rn)]^{\frac{1}{q_0}})'$.
Indeed, from \cite[Corollary 1.5.5]{LYH2022}, we deduce that, to obtain this desired
boundedness, it suffices to show
\begin{align}\label{lyq-X-Y-11}
\lf(\frac{np}{[n-\alpha p]q_0}\r)'>1,
\end{align}
\begin{align}\label{lyq-X-Y-09}
-\frac{n}{(\frac{np}{[n-\alpha p]q_0})'}<m_0(\omega^{-q_0})
\leq M_0(\omega^{-q_0})<\frac{(n-\alpha p)q_0}{p},
\end{align}
and
\begin{align}\label{lyq-X-Y-10}
-\frac{n}{(\frac{np}{[n-\alpha p]q_0})'}<m_\fz(\omega^{-q_0})
\leq M_\fz(\omega^{-q_0})<\frac{(n-\alpha p)q_0}{p}.
\end{align}
Notice that
\eqref{lyq-X-Y-11} follows from both the assumption $p_0<p$
and $\frac{1}{q_0}=\frac{1}{p_0}-\frac{\alpha}{n}$.
In addition, by \eqref{q014}
and the assumption that $\min\{m_0(\omega),m_{\fz}(\rn)\}>\alpha-\frac{n}{p}$,
we conclude that
\begin{align}\label{lyq-X-Y-05}
-\frac{n}{(\frac{np}{[n-\alpha p]q_0})'}<-q_0M_0(\omega)\leq -q_0m_0(\omega)
<\frac{(n-\alpha p)q_0}{p}
\end{align}
and
\begin{align}\label{lyq-X-Y-06}
-\frac{n}{(\frac{np}{[n-\alpha p]q_0})'}<
-q_0M_\fz(\omega)\leq -q_0m_\fz(\omega)
<\frac{(n-\alpha p)q_0}{p}.
\end{align}
Moreover, using \cite[p.\,6, Lemma 1.1.6]{LYH2022}, we have
$m_0(\omega^{-q_0})=-q_0M_0(\omega)$,
$M_0(\omega^{-q_0})=-q_0m_0(\omega)$,
$m_\fz(\omega^{-q_0})=-q_0M_\fz(\omega)$, and
$M_\fz(\omega^{-q_0})=-q_0m_\fz(\omega)$. From this,
\eqref{lyq-X-Y-05}, and \eqref{lyq-X-Y-06}, we deduce that both
\eqref{lyq-X-Y-09} and \eqref{lyq-X-Y-10} hold true
and hence $\mathcal{M}$ is bounded on
$([\dot{\mathcal{K}}_{\omega,\mathbf{0}}^
{\frac{np}{n-\alpha p},\frac{nq}{n-\alpha q}}(\rn)]^{\frac{1}{q_0}})'$.

All together,
we find that all the assumptions of Theorem \ref{thm-R}
with both $X:=\dot{\mathcal{K}}_{\omega,\mathbf{0}}^{p,q}(\rn)$
and $Y:=\dot{\mathcal{K}}_{\omega,\mathbf{0}}^{\frac{np}{n-\alpha
p},\frac{nq}{n-\alpha q}}(\rn)$ are satisfied.
Thus, by Theorem \ref{thm-R} with both $X:=\dot{\mathcal{K}}_{\omega,\mathbf{0}}^{p,q}(\rn)$
and $Y:=\dot{\mathcal{K}}_{\omega,\mathbf{0}}^{\frac{np}{n-\alpha
p},\frac{nq}{n-\alpha q}}(\rn)$, we obtain the desired result.
This finishes the proof of Theorem \ref{lyq-X-Y-Z}.
\end{proof}

\begin{remark}
Let $\alpha\in(0,n)$, $I_\alpha$ be the same as in \eqref{cla-I}, $\omega(t):=t^{\wz \alpha}$ for any $t\in(0,\infty)$ with $\wz \alpha\in(0,\infty)$, and
$\dot{\mathcal{K}}_{\wz\alpha, \mathbf{0}}^{p,q}(\rn):=
\dot{\mathcal{K}}_{\omega, \mathbf{0}}^{p,q}(\rn)$.
Lu et al. \cite[Theorem 2.6]{LY1996} proved that $I_\alpha$ is bounded from
$H\dot{\mathcal{K}}_{\wz\alpha, \mathbf{0}}^{p_1,q_1}(\rn)$
to
$H\dot{\mathcal{K}}_{\wz\alpha, \mathbf{0}}^{p_2,q_2}(\rn)$ when
$p_1\in(1,\frac{n}{\alpha})$, $\frac{1}{p_2}=\frac{1}{p_1}
-\frac{\alpha}{n}$, $0<q_1\leq q_2<\infty$,
and $\wz\alpha\in(n(1-\frac{1}{p_2}),\infty)$.
\begin{enumerate}
\item[$\mathrm{(i)}$]
From Theorem \ref{Th-LGH}, we deduce
that $I_\alpha$ is bounded
from $H\dot{\mathcal{K}}_{\wz\alpha,
\mathbf{0}}^{p,q}(\rn)$ to
$H\dot{\mathcal{K}}_{\wz\alpha/\beta,
\mathbf{0}}^{\beta p,\beta q}(\rn)$
with $\beta$ the same as in Theorem \ref{Th-LGH}.
Thus, we conclude that Theorem \ref{Th-LGH} and \cite[Theorem 2.6]{LY1996} can not
cover each other
even in the case that $\omega(t):=t^{\wz \alpha}$ for any $t\in(0,\infty)$.
\item[$\mathrm{(ii)}$]
By Theorem \ref{lyq-X-Y-Z}, we infer that $I_\alpha$ is bounded from
$H\dot{\mathcal{K}}_{\wz\alpha,\mathbf{0}}^{p,q}(\rn)$ to
$H\dot{\mathcal{K}}_{\wz\alpha,\mathbf{0}}^{\frac{np}{n-\alpha p},\frac{nq}{n-\alpha q}}(\rn)$ with
$p,q\in(0,\frac{n}{\alpha})$ and $\wz\alpha\in(\alpha-\frac{n}{p},\infty)$.
When $\omega(t):=t^{\wz \alpha}$ for any $t\in(0,\infty)$,
$p\in(1,\frac{n}{\alpha})$, and $\wz\alpha\in(n+\alpha-\frac{n}{p},\infty)$,
it is easy to show that \cite[Theorem 2.6]{LY1996} is more general than Theorem \ref{lyq-X-Y-Z} in this case.
However, even in the case that $\omega(t):=t^{\wz \alpha}$ for any
$t\in(0,\infty)$, Theorem \ref{lyq-X-Y-Z} can be applied to
$H\dot{\mathcal{K}}_{\wz\alpha,\mathbf{0}}^{p,q}(\rn)$
when $p\in(0,1]$ or
$\wz\alpha\in(\alpha-\frac{n}{p},n+\alpha-\frac{n}{p})$,
but \cite[Theorem 2.6]{LY1996} can not.
\end{enumerate}
\end{remark}

\subsection{Mixed Herz Spaces\label{M-H}}

We first recall the following definition of
mixed Herz spaces, which is just \cite[Definition 2.3]{zyz2022}.

\begin{definition}\label{mhz}
Let $\vec{p}:=(p_{1},\ldots,p_{n}),\vec{q}:=(q_{1},\ldots,q_{n})
\in(0,\infty]^{n}$, $\vec{\alpha}:=
(\alpha_{1},\ldots,
\alpha_{n})\in\rn$, and $R_{k_i}:=(-2^{k_i},2^{k_i})\setminus
(-2^{k_i-1},2^{k_i-1})$ for any $i\in\{1,\ldots,n\}$ and
$k_i\in\zz$.
The \emph{mixed Herz space}
$\dot{E}^{\vec{\alpha},\vec{p}}_{\vec{q}}(\rn)$ is
defined to be the
set of all the functions
$f\in \mathcal{M}(\rn)$ such
that
\begin{align*}
\|f\|_{\dot{E}^{\vec{\alpha},\vec{p}}_{\vec{q}}
(\rn)}:&=\lf\{\sum_{k_{n} \in
\zz}2^{k_{n}
p_{n}\alpha_{n}}
\lf\{\int_{R_{k_{n}}}\cdots\Bigg\{\sum_{k_{1}
\in \zz}
2^{k_{1}p_{1}\alpha_{1}}\r.\r.\\
&\lf.\lf.\lf.\quad\times\lf[\int_{R_{k_{1}}}|f(x_{1},
\ldots,x_{n})|
^{q_{1}}\,dx_{1} \r]^{\f{p_{1}}{q_{1}}}
\r\}^{\f{q_{2}}
{p_{1}}}\cdots
\,dx_{n}  \r\}^{\f{p_{n}}{q_{n}}}\r\}
^{\f{1}{p_n}}\\
&=:\,\|\cdots\|f\|_{\dot{K}^{\alpha_{1},p_{1}}_{q_{1}}
(\rr)}\cdots\|_{\dot{K}^{\alpha_{n},p_{n}}_{q_{n}}(\rr)}
<\infty
\end{align*}
with the usual modifications made when $p_{i}
=\infty$ or
$q_{j}=\infty$ for some $i,j\in\{1,\ldots,n \}$,
where
we denote by $\|\cdots\|f\|_{\dot{K}^{\alpha_{1},p_{1}}_{q_{1}}(\rr)}
\cdots\|_{\dot{K}^{\alpha_{n},p_{n}}_{q_{n}}(\rr)}$ the
norm obtained after taking successively
the $\dot{K}^{\alpha_{1},p_{1}}_{q_{1}}(\rr)$ norm
to $x_{1}$, the $\dot{K}^{\alpha_{2},p_{2}}_{q_{2}}
(\rr)$ norm to $x_{2}$,
$\ldots$, and
the $\dot{K}^{\alpha_{n},p_{n}}_{q_{n}}(\rr)$ norm to $x_{n}$.
\end{definition}

Recall that a special case of the mixed Herz space was originally
introduced by Huang et al. \cite{HWYY2021} to study the Lebesgue
points of functions in mixed-norm Lebesgue spaces and, later,
Zhao et al. \cite{zyz2022} generalized it to the above case.
Also, both the dual theorem and the Riesz--Thorin
interpolation theorem on the above mixed Herz space
have been fully studied in \cite{zyz2022}.

\begin{remark}
Let $\vec{p},\vec{q}
\in(0,\infty]^{n}$ and $\vec{\alpha}\in\rn$.
By \cite[Propositions 2.8 and 2.22]{zyz2022}, we conclude that
$\dot{E}^{\vec{\alpha},\vec{p}}_{\vec{q}}(\rn)$
is a ball quasi-Banach space.
However, from \cite[Remark 2.4]{zyz2022}, we deduce that, when $\vec{p}=\vec{q}$ and
$\vec\alpha=\mathbf{0}$, the mixed Herz
space $\dot{E}^{\vec{\alpha},\vec{p}}_{\vec{q}}(\rn)$
coincides with the mixed Lebesgue
space $L^{\vec{p}}(\rn)$ defined in Definition \ref{mixed}.
Using Remark \ref{mix-r}, we conclude that $L^{\vec{p}}(\rn)$
may not be a quasi-Banach function space and hence
$\dot{E}^{\vec{\alpha},\vec{p}}_{\vec{q}}(\rn)$
may not be a quasi-Banach function space.
\end{remark}

The following theorem is a corollary of Theorem \ref{thm-Ia-02}. In
what follows, the
\emph{mixed Herz-Hardy space}
$H\dot{E}^{\vec{\alpha},\vec{p}}_{\vec{q}}(\rn)$ is defined as in
Definition \ref{2d1} with
$X:=\dot{E}^{\vec{\alpha},\vec{p}}_{\vec{q}}(\rn)$
(see \cite[Definition 5.2]{zyz2022}).

\begin{theorem}\label{Th-M-H}
Let $\beta\in(1,\infty)$, $\alpha\in(0,n)$, and $I_\alpha$ be the same as in \eqref{cla-I}.
Let $\vec{p}:=(p_{1},\ldots,p_{n}),\vec{q}:=(q_{1},\ldots,q_{n})
\in(0,\infty)^{n}$, and $\vec{\alpha}:=
(\alpha_{1},\ldots,
\alpha_{n})\in\rn$ satisfy
$\sum_{i=1}^{n}(\frac{1}{p_i}+\alpha_i)\in(\alpha,\infty)$
and $\alpha_i\in(0,\infty)$ for any
$i\in\{1,\ldots,n\}$.
Then $I_\alpha$ can be extended to a unique bounded linear
operator, still denoted by $I_\alpha$, from
$H\dot{E}^{\vec{\alpha},\vec{p}}_{\vec{q}}(\rn)$ to
$H\dot{E}^{\vec{\alpha}/\beta,\beta\vec{p}}_{\beta\vec{q}}(\rn)$,
namely, there exists a positive constant $C$ such that, for any $f\in
H\dot{E}^{\vec{\alpha},\vec{p}}_{\vec{q}}(\rn)$,
$$\|I_\alpha(f)\|_{H\dot{E}^{\vec{\alpha}/\beta,\beta\vec{p}}_{\beta\vec{q}}(\rn)}\leq
C\|f\|_{H\dot{E}^{\vec{\alpha},\vec{p}}_{\vec{q}}(\rn)}
$$
if and only if $\beta=\frac{\sum_{i=1}^{n}(\frac{1}{p_i}+\alpha_i)}
{\sum_{i=1}^{n}(\frac{1}{p_i}+\alpha_i)-\alpha}$.
\end{theorem}

\begin{proof}
Let all the symbols be the same as in the present
theorem. By \cite[Propositions 2.8 and 2.22]{zyz2022}, we conclude that
$\dot{E}^{\vec{\alpha},\vec{p}}_{\vec{q}}(\rn)$
is a ball quasi-Banach space with an absolutely continuous quasi-norm.
From \cite[Lemma 5.3(i)]{zyz2022} and its proof,
we infer that
$\dot{E}^{\vec{\alpha},\vec{p}}_{\vec{q}}(\rn)$ satisfies Assumption
\ref{assump1} with
\begin{align}\label{p-zyr}
p_-:=\min\lf\{p_1,\ldots,p_n,
q_1,\ldots,q_n,\lf(\alpha_{1}+\frac{1}{q_{1}}\r)^{-1},
\ldots,\lf(\alpha_{n}+\frac{1}{q_{n}}\r)^{-1} \r \}.
\end{align}
Let $$r_0\in\lf(0,\min\lf\{\frac{1}{\beta},p_-\r\}\r)$$
and $$p_0\in\lf(\max\lf\{p_1,\ldots,p_n,
q_1,\ldots,q_n,\lf(\alpha_{1}
+\frac{1}{q_{1}}\r)^{-1}
,\ldots,\lf(\alpha_{n}+\frac
{1}{q_{n}}\r)^{-1}\r\},\infty \r).$$ Then, by \cite[Lemma 5.3(ii)]{zyz2022}
and its proof,
we conclude that Assumption \ref{assump2} also
holds true with $X:=\dot{E}^{\vec{\alpha},\vec{p}}_{\vec{q}}(\rn)$.
Thus, all the assumptions of Theorem \ref{thm-Ia-02}
with $X:=\dot{E}^{\vec{\alpha},\vec{p}}_{\vec{q}}(\rn)$ hold true.

Moreover, for any
$B:=B(x,r)\in\mathbb{B}(\rn)$, where
$x=(x_1,x_2,\dots,x_n)\in\rn$ and
$r\in(0,\infty)$, letting $Q(x,r)$ be the
cube with edges parallel to the coordinate axes, center $x\in\rn$,
edge length
$r$, by \cite[(4.9.12)]{LYH2022}
with $\omega(t):=t^{\alpha_i}$ for any $t\in(0,\infty)$ and $i\in\{1,\ldots,n\}$, we have
\begin{align*}
\lf\|\mathbf{1}_B\r\|_{\dot{E}^{\vec{\alpha},\vec{p}}_{\vec{q}}(\rn)}
&\gtrsim\|\mathbf{1}_{Q(x,r)}\|_{\dot{E}^{\vec{\alpha},\vec{p}}
_{\vec{q}}(\rn)}\\
&\sim\Pi_{i=1}^n\|\mathbf{1}_{(x_i-r,x_i+r)}\|_{\dot{K}
^{\alpha_{i},p_{i}}_{q_{i}}
(\rr)}
\gtrsim r^{\sum_{i=1}^{n}(\frac{1}{p_i}+\alpha_i)}.
\end{align*}
This implies that, for any ball $B\in\mathbb{B}(\rn)$,
\begin{align}\label{zyrB}
|B|^{\frac{\alpha}{n}}\lesssim\|\mathbf{1}_B\|_
{\dot{E}^{\vec{\alpha},\vec{p}}_{\vec{q}}(\rn)}^{\frac{\beta-1}{\beta}}
\ \text{if and only if}\
\beta=\frac{\sum_{i=1}^{n}(\frac{1}{p_i}+\alpha_i)}
{\sum_{i=1}^{n}(\frac{1}{p_i}+\alpha_i)-\alpha}.
\end{align}
Then, using Theorem \ref{thm-Ia-02} with
$X:=\dot{E}^{\vec{\alpha},\vec{p}}_{\vec{q}}(\rn)$,
we obtain the desired conclusion,
which completes the proof of
Theorem \ref{Th-M-H}.
\end{proof}

Using Theorem \ref{coclassic}, we can obtain the following conclusion.

\begin{theorem}\label{co-M-H}
Let all the symbols be the same as in Theorem \ref{Th-M-H},
$p_-\in(1,\infty)$ the same as in \eqref{p-zyr}, and $\mathcal{M}_\alpha$
the same as in \eqref{de-Ma}.
Then $\mathcal{M}_\alpha$ is bounded from
$\dot{E}^{\vec{\alpha},\vec{p}}_{\vec{q}}(\rn)$ to
$\dot{E}^{\vec{\alpha}/\beta,\beta\vec{p}}_{\beta\vec{q}}(\rn)$,
namely, there exists a positive constant $C$ such that, for any $f\in
\dot{E}^{\vec{\alpha},\vec{p}}_{\vec{q}}(\rn)$,
$$\|\mathcal{M}_\alpha(f)\|_{\dot{E}^{\vec{\alpha}/\beta,\beta\vec{p}}_{\beta\vec{q}}(\rn)}\leq
C\|f\|_{\dot{E}^{\vec{\alpha},\vec{p}}_{\vec{q}}(\rn)}
$$
if and only if $\beta=\frac{\sum_{i=1}^{n}(\frac{1}{p_i}+\alpha_i)}
{\sum_{i=1}^{n}(\frac{1}{p_i}+\alpha_i)-\alpha}$.
\end{theorem}

\begin{proof}
Let all the symbols be the same as in the present theorem.
From both $p_-\in(1,\infty)$ and Remark \ref{max-re}, we infer that
$\dot{E}^{\vec{\alpha},\vec{p}}_{\vec{q}}(\rn)$ satisfies Assumption \ref{max-assump}.
Thus, all the assumptions of Theorem \ref{coclassic}
with $X:=\dot{E}^{\vec{\alpha},\vec{p}}_{\vec{q}}(\rn)$
are satisfied.
Then, using \eqref{zyrB} and Theorem \ref{coclassic} with
$X:=\dot{E}^{\vec{\alpha},\vec{p}}_{\vec{q}}(\rn)$,
we obtain the desired conclusion, which completes the proof of Theorem \ref{co-M-H}.
\end{proof}

The following theorem is a corollary of
Theorem \ref{thm-R}.

\begin{theorem}\label{lyq-X}
Let $\alpha\in(0,n)$,
$\vec{p}:=(p_{1},\ldots,p_{n}),\vec{q}:=(q_{1},\ldots,q_{n})
\in(0,\frac{n}{\alpha})^{n}$,
$\vec{\alpha}:=
(\alpha_{1},\ldots,
\alpha_{n})\in\rn$ satisfy
$\alpha_i\in(\frac{\alpha q_i-n}{nq_i},\infty)$ for any
$i\in\{1,\ldots,n\}$, $\vec{r}:=(\frac{np_1}{n-\alpha
p_1},\ldots,\frac{np_n}{n-\alpha p_n})$,
$\vec{s}:=(\frac{nq_1}{n-\alpha q_1},\ldots,\frac{nq_n}{n-\alpha q_n})$,
and $I_\alpha$ be the same as in \eqref{cla-I}.
Then $I_\alpha$ can be extended to a unique bounded linear
operator, still denoted by $I_\alpha$, from
$H\dot{E}^{\vec{\alpha},\vec{p}}_{\vec{q}}(\rn)$ to
$H\dot{E}^{\vec{\alpha},\vec{r}}_{\vec{s}}(\rn)$,
namely, there exists a positive constant $C$ such that, for any $f\in
H\dot{E}^{\vec{\alpha},\vec{p}}_{\vec{q}}(\rn)$,
$$\|I_\alpha(f)\|_{H\dot{E}^{\vec{\alpha},\vec{r}}_{\vec{s}}(\rn)}\leq
C\|f\|_{H\dot{E}^{\vec{\alpha},\vec{p}}_{\vec{q}}(\rn)}.
$$
\end{theorem}

\begin{proof}
Let all the symbols be the same as in the present theorem.
Let
$$p_0\in\lf(0,\min\lf\{\frac{n}{n+\alpha},
p_-\r\}\r)
$$
with $p_-$ the same as in \eqref{p-zyr}
and $\frac{1}{q_0}:=\frac{1}{p_0}-\frac{\alpha}{n}$.
From the range of $\alpha_i$, we deduce that $p_0\alpha_i\in(-\frac{p_0}{q_i},1-\frac{p_0}{q_i})$
and
$q_0\alpha_i\in(\frac{q_0}{s_i},1+\frac{q_0}{s_i})$
for any $i\in\{1,\ldots,n\}$.
By this and \cite[Proposition 2.22]{zyz2022}, we conclude that both $\dot{E}^{p_0\vec{\alpha},\vec{p}/p_0}_{\vec{q}/p_0}(\rn)$ and $\dot{E}^{q_0\vec{\alpha},\vec{r}/q_0}_{\vec{s}/q_0}(\rn)$
are ball Banach function spaces.
Moreover, using \cite[Lemma 2.6 and Theorem 2.10]{zyz2022}, we have
$$\lf(\lf\{\lf[\dot{E}^{\vec{\alpha},\vec{p}}_{\vec{q}}(\rn)
\r]^{\frac{1}{p_0}}\r\}'\r)^{\frac{p_0}{q_0}}
=\dot{E}^{-q_0\vec{\alpha},\frac{p_0}{q_0}(\frac{\vec{p}}{p_0})'}
_{\frac{p_0}{q_0}(\frac{\vec{q}}{p_0})'}(\rn)
=\lf(\lf[\dot{E}^{\vec{\alpha},\vec{r}}_{\vec{s}}(\rn)
\r]^{\frac{1}{q_0}}\r)'.
$$
Thus, \eqref{thm-R-02} with both
$X:=\dot{E}^{\vec{\alpha},\vec{r}}_{\vec{s}}(\rn)$
and $Y:=\dot{E}^{\vec{\alpha},\vec{p}}_{\vec{q}}(\rn)$
holds true.
In addition, from the range of $p_0$ and $\alpha_i$,
we infer that $\frac{p_0}{q_0}(\frac{p_i}{p_0})'\in(1,\infty)$,
$\frac{p_0}{q_0}(\frac{q_i}{p_0})'\in(1,\infty)$, and $-q_0\vec{\alpha}\in(
\frac{q_0(p_0-q_i)}{p_0q_i},
1+\frac{q_0(p_0-q_i)}{p_0q_i})$ for any $i\in\{1,\ldots,n\}$. Combining this and
\cite[Corollary 4.9]{zyz2022}, we conclude that Assumption \ref{assump1} with $X:=(\dot{E}^{\vec{\alpha},\vec{r}}_{\vec{s}}(\rn)
^{\frac{1}{q_0}})'$ holds true.
Moreover, by the range of $p_0$, $\vec{\alpha}$, $\vec{p}$, and $\vec{q}$,
we obtain that $-q_0\alpha_i+[\frac{p_0}{q_0}(\frac{q_i}{p_0})']^{-1}\in(0,1)$.
All together, by Remark \ref{max-re}, we find that
$\mathcal{M}$ is bounded on $(\dot{E}^{\vec{\alpha},\vec{r}}_{\vec{s}}(\rn)
^{\frac{1}{q_0}})'$.
Thus, all the assumptions of Theorem \ref{thm-R} with both
$X:=\dot{E}^{\vec{\alpha},\vec{r}}_{\vec{s}}(\rn)$
and $Y:=\dot{E}^{\vec{\alpha},\vec{p}}_{\vec{q}}(\rn)$
are satisfied.
Then, using Theorem \ref{thm-R} with both
$X:=\dot{E}^{\vec{\alpha},\vec{r}}_{\vec{s}}(\rn)$
and $Y:=\dot{E}^{\vec{\alpha},\vec{p}}_{\vec{q}}(\rn)$, we obtain the desired conclusion, which completes
the proof of Theorem \ref{lyq-X}.
\end{proof}

\begin{remark}
To the best of our knowledge, Theorems \ref{Th-M-H}, \ref{co-M-H},
and \ref{lyq-X} are totally new.
\end{remark}

\bigskip

\noindent Yiqun Chen, Hongchao Jia and Dachun Yang (Corresponding
author)

\smallskip

\noindent  Laboratory of Mathematics and Complex Systems
(Ministry of Education of China),
School of Mathematical Sciences, Beijing Normal University,
Beijing 100875, The People's Republic of China

\smallskip

\noindent {\it E-mails}: \texttt{yiqunchen@mail.bnu.edu.cn} (Y. Chen)

\noindent\phantom{{\it E-mails:}} \texttt{hcjia@mail.bnu.edu.cn} (H. Jia)

\noindent\phantom{{\it E-mails:}} \texttt{dcyang@bnu.edu.cn} (D. Yang)


\begin{thebibliography}{99}

\bibitem{ADR1975}
D. R. Adams, A note on Riesz potentials, Duke Math. J. 42 (1975), 765-778.

\vspace{-0.3cm}

\bibitem{ADR2015}
D. R. Adams, Morrey Spaces,
Lecture Notes in Applied and Numerical Harmonic Analysis,
Birkh\"auser/Springer, Cham, 2015.

\vspace{-0.3cm}

\bibitem{ans2021}
R. Arai, E. Nakai and Y. Sawano,
Generalized fractional integral operators on Orlicz--Hardy spaces,
Math. Nachr. 294 (2021), 224-235.

\vspace{-0.3cm}

\bibitem{BAP1961}
A. Benedek and R. Panzone, The space $L^p$, with mixed norm,
Duke Math. J. 28 (1961), 301-324.

\vspace{-0.3cm}

\bibitem{BS88}
C. Bennett and R. Sharpley, Interpolation of Operators,
Pure and Applied Mathematics 129, Academic Press, Boston, MA, 1988.

\vspace{-0.3cm}

\bibitem{dfmn2021}
R. Del Campo, A. Fern\'andez, F. Mayoral and F. Naranjo,
Orlicz spaces associated to a quasi-Banach function space:
applications to vector measures and interpolation,
Collect. Math. 72 (2021), 481-499.

\vspace{-0.3cm}

\bibitem{ccy2013}
J. Cao, D.-C. Chang, D. Yang and S. Yang,
Boundedness of fractional integrals on weighted Orlicz--Hardy spaces,
Math. Methods Appl. Sci. 36 (2013), 2069-2085.

\vspace{-0.3cm}

\bibitem{CWYZ2020}
D.-C. Chang, S. Wang, D. Yang and Y. Zhang,
Littlewood--Paley characterizations of Hardy-type
spaces associated with ball quasi-Banach function spaces,
Complex Anal. Oper. Theory 14 (2020), Paper No. 40, 33 pp.

\vspace{-0.3cm}

\bibitem{cs2022}
T. Chen and W. Sun,
Hardy--Littlewood--Sobolev inequality on mixed-norm Lebesgue spaces,
J. Geom. Anal. 32 (2022), Paper No. 101, 43 pp.

\vspace{-0.3cm}

\bibitem{chy2021}
K. L. Cheung, K.-P. Ho and T.-L. Yee,
Boundedness of fractional integral operators on Hardy-amalgam spaces,
J. Funct. Spaces 2021, Art. ID 1142942, 5 pp.

\vspace{-0.3cm}

\bibitem{clt2019}
L. Chen, G. Lu and C. Tao,
Hardy--Littlewood--Sobolev inequalities with the fractional
Poisson kernel and their applications in PDEs,
Acta Math. Sin. (Engl. Ser.) 35 (2019), 853-875.

\vspace{-0.3cm}

\bibitem{cjy2022}
Y. Chen, H. Jia and D. Yang,
Boundedness of fractional integrals and Calder\'on--Zygmund
operators on ball Campanato-type function spaces, Submitted.

\vspace{-0.3cm}

\bibitem{CFF1987}
F. Chiarenza and M. Frasca,
Morrey spaces and Hardy--Littlewood maximal function, Rend. Mat.
Appl. (7) 7 (1987), 273-279 (1988).

\vspace{-0.3cm}

\bibitem{CGN2017}
G. Cleanthous, A. G. Georgiadis and M. Nielsen,
Anisotropic mixed-norm Hardy spaces,
J. Geom. Anal. 27 (2017), 2758-2787.

\vspace{-0.3cm}

\bibitem{CGG2017}
G. Cleanthous, A. G. Georgiadis and M. Nielsen, Discrete
decomposition of homogeneous mixed-norm Besov spaces, in: Functional Analysis, Harmonic Analysis, and Image
Processing: A Collection of Papers in Honor of Bj\"orn Jawerth,
167-184,
Contemp. Math. 693, Amer. Math. Soc., Providence, RI, 2017.

\vspace{-0.3cm}

\bibitem{CGN20172}
G. Cleanthous, A. G. Georgiadis and M. Nielsen,
Molecular decomposition of anisotropic homogeneous mixed-norm spaces with
applications to the boundedness of operators,
Appl. Comput. Harmon. Anal. 47 (2019), 447-480.

\vspace{-0.3cm}

\bibitem{CW}
R. R. Coifman and G. Weiss,
Extensions of Hardy spaces and their use in analysis,
Bull. Amer. Math. Soc. 83 (1977), 569-645.

\vspace{-0.3cm}

\bibitem{dsl}
C. Deng, J. Sun and B. Li, Off-diagonal extrapolation
on ball Banach function
spaces and applications, Submitted.

\vspace{-0.3cm}

\bibitem{DGNSS}
F. Deringoz, V. S. Guliyev, E. Nakai, Y. Sawano and M. Shi,
Generalized fractional maximal and integral operators on Orlicz and generalized Orlicz-Morrey spaces of the third kind,
Positivity 23 (2019), 727-757.

\vspace{-0.3cm}

\bibitem{dll2003}
Y. Ding, M. Y. Lee and C. C. Lin,
Fractional integrals on weighted Hardy spaces,
J. Math. Anal. Appl. 282 (2003), 356-368.

\vspace{-0.3cm}

\bibitem{Duo01}
J. Duoandikoetxea, Fourier Analysis, Graduate Studies
in Mathematics 29, American
Mathematical Society, Providence, RI, 2001.

\vspace{-0.3cm}

\bibitem{FS72}
C. Fefferman and E. M. Stein,
$H^p$ spaces of several variables,
Acta Math. 129 (1972), 137-193.

\vspace{-0.3cm}

\bibitem{GN2016}
A. G. Georgiadis and M. Nielsen,
Pseudodifferential operators on mixed-norm Besov and Triebel--Lizorkin spaces,
Math. Nachr. 289 (2016), 2019-2036.

\vspace{-0.3cm}

\bibitem{G1}
L. Grafakos, Classical Fourier Analysis, Third edition, Graduate Texts in Mathematics 249, Springer, New York, 2014.

\vspace{-0.3cm}

\bibitem{GD2020}
V. S. Guliyev and F. Deringoz, Riesz potential and
its commutators on generalized
weighted Orlicz--Morrey spaces, Math. Nachr. 295 (2022), 706-724.

\vspace{-0.3cm}

\bibitem{HSS2016}
D. I. Hakim, Y. Sawano and T. Shimomura,
Boundedness of generalized fractional integral operators from the Morrey space
$L_{1,\phi}(X;\mu)$ to the Campanato space
$L_{1,\psi}(X;\mu)$ over non-doubling measure spaces,
Azerb. J. Math. 6 (2016), 117-127.

\vspace{-0.3cm}

\bibitem{hl1928}
H. G. Hardy and J. E. Littlewood,
Some properties of fractional integrals. I,
Math. Z. 27 (1928), 565-606.

\vspace{-0.3cm}

\bibitem{ho2020}
K.-P. Ho,
Martingale transforms and fractional integrals on
rearrangement-invariant martingale Hardy spaces,
Period. Math. Hungar. 81 (2020), 159-173.

\vspace{-0.3cm}

\bibitem{Ho2020}
K.-P. Ho,
Sublinear operators on Herz--Hardy spaces with variable
exponents, Math. Nachr. 295 (2022),
876-889.

\vspace{-0.3cm}

\bibitem{ho2021}
K.-P. Ho,
Erd\'elyi-Kober fractional integral operators on ball Banach function spaces,
Rend. Semin. Mat. Univ. Padova 145 (2021), 93-106.

\vspace{-0.3cm}

\bibitem{Ho}
K.-P. Ho, Fractional integral operators on Orlicz
slice Hardy spaces, Fract. Calc. Appl. Anal. (2022), https://doi.org/10.1007/s13540-022-00043-1.

\vspace{-0.3cm}

\bibitem{HL1960}
L. H\"ormander, Estimates for translation invariant operators in
$L^p$ spaces, Acta Math. 104 (1960), 93-140.

\vspace{-0.3cm}

\bibitem{HLY2019}
L. Huang, J. Liu, D. Yang and W. Yuan, Atomic and Littlewood--Paley
characterizations of anisotropic
mixed-norm Hardy spaces and their applications,
J. Geom. Anal. 29 (2019), 1991-2067.

\vspace{-0.3cm}

\bibitem{HLYY2019}
L. Huang, J. Liu, D. Yang and W. Yuan, Dual spaces of anisotropic
mixed-norm Hardy spaces, Proc. Amer. Math. Soc. 147 (2019),
1201-1215.

\vspace{-0.3cm}

\bibitem{HWYY2021}
L. Huang, F. Weisz, D. Yang and W. Yuan, Summability of Fourier transforms on
mixed-norm Lebesgue spaces via associated Herz spaces, Anal. Appl. (Singap.) (2021),
http://doi.org/10.1142/S0219530521500135.

\vspace{-0.3cm}

\bibitem{HYacc}
L. Huang and D. Yang, On function spaces with mixed norms--a survey,
J. Math. Study 54 (2021), 262-336.

\vspace{-0.3cm}

\bibitem{hk2021}
D. Q. Huy and L. D. Ky,
Boundedness of fractional integral operators on
Musielak--Orlicz Hardy spaces,
Math. Nachr. 294 (2021), 2340-2354.

\vspace{-0.3cm}

\bibitem{is2017}
M. Izuki and Y. Sawano, Characterization of BMO via
ball Banach function spaces,
Vestn. St.-Peterbg. Univ. Mat. Mekh. Astron. 4(62) (2017), 78-86.

\vspace{-0.3cm}

\bibitem{JHW2009}
H. Jia and H. Wang, Decomposition of Hardy--Morrey spaces, J. Math.
Anal. Appl. 354 (2009), 99-110.

\vspace{-0.3cm}

\bibitem{JN}
F. John and L. Nirenberg,
On functions of bounded mean oscillation,
Comm. Pure Appl. Math. 14 (1961), 415-426.

\vspace{-0.3cm}

\bibitem{LPI1970}
P. I. Lizorkin, Multipliers of Fourier integrals and estimates of
convolutions in spaces with mixed norm, Applications, (Russian)
Izv. Akad. Nauk SSSR Ser. Mat. 34 (1970), 218-247.

\vspace{-0.3cm}

\bibitem{LYH2022}
Y. Li, D. Yang and L. Huang, Real-variable Theory of
Hardy Spaces Associated with Generalized Herz
Spaces of Rafeiro and Samko,
Lecture Notes in Mathematics, Springer, Cham, 2022 (to appear) or arXiv: 2203.15165.

\vspace{-0.3cm}

\bibitem{Lu}
S. Lu, Four Lectures on Real $H^p$ Spaces,
World Scientific Publishing Co., River Edge, NJ, 1995.

\vspace{-0.3cm}

\bibitem{LDY2007}
S. Lu, Y. Ding and D. Yan, Singular Integrals and Related Topics,
World Scientific Publishing Co. Pte. Ltd., Hackensack, NJ, 2007.

\vspace{-0.3cm}

\bibitem{LY1996}
S. Lu and D. Yang, Hardy--Littlewood--Sobolev theorems of
fractional integration on Herz-type spaces and its applications,
Canad. J. Math. 48 (1996), 363-380.

\vspace{-0.3cm}

\bibitem{mv1995}
V. G. Maz'ya and I. E. Verbitsky,
Capacitary inequalities for fractional integrals,
with applications to partial differential
equations and Sobolev multipliers,
Ark. Mat. 33 (1995), 81-115.

\vspace{-.3cm}

\bibitem{MCB1938}
C. B. Morrey, On the solutions of quasi-linear elliptic partial
differential equations,
Trans. Amer. Math. Soc. 43 (1938), 126-166.

\vspace{-.3cm}

\bibitem{N10}
E. Nakai,
Singular and fractional integral operators on Campanato
spaces with variable growth conditions,
Rev. Mat. Complut. 23 (2010), 355-381.

\vspace{-.3cm}

\bibitem{N17}
E. Nakai,
Singular and fractional integral operators on preduals of
Campanato spaces with variable growth condition,
Sci. China Math. 60 (2017), 2219-2240.

\vspace{-0.3cm}

\bibitem{NT2019}
T. Nogayama, Mixed Morrey spaces, Positivity 23 (2019), 961-1000.

\vspace{-.3cm}

\bibitem{Po1997}
I. Podlubny,
Riesz potential and Riemann--Liouville
fractional integrals and derivatives of Jacobi polynomials,
Appl. Math. Lett. 10 (1997), 103-108.

\vspace{-.3cm}

\bibitem{RS2020}
H. Rafeiro and S. Samko,
Herz spaces meet Morrey type spaces and complementary Morrey type spaces,
J. Fourier Anal. Appl. 26 (2020), Paper No. 74, 14 pp.

\vspace{-.3cm}

\bibitem{R1996}
B. Rubin, Fractional Integrals and Potentials,
Pitman Monographs and Surveys in Pure and Applied
Mathematics 82, Longman, Harlow, 1996.

\vspace{-.3cm}

\bibitem{Ru1982}
J. L. Rubio de Francia,
Factorization and extrapolation of weights,
Bull. Amer. Math. Soc. (N.S.) 7 (1982), 393-395.

\vspace{-0.3cm}

\bibitem{SHS}
Y. Sawano, D. I. Hakim and D. Salim,
Riesz transform and fractional integral operators
generated by nondegenerate elliptic differential operators,
Adv. Oper. Theory 4 (2019), 750-766.

\vspace{-.3cm}

\bibitem{SHYY}
Y. Sawano, K.-P. Ho, D. Yang and S. Yang,
Hardy spaces for ball quasi-Banach function spaces,
Dissertationes Math. 525 (2017), 1-102.

\vspace{-.3cm}

\bibitem{SS2017}
Y. Sawano and T. Shimomura,
Boundedness of the generalized fractional integral operators on
generalized Morrey spaces over metric measure spaces,
Z. Anal. Anwend. 36 (2017), 159-190.

\vspace{-.3cm}

\bibitem{SST2009}
Y. Sawano, S. Sugano and H. Tanaka,
A note on generalized fractional integral operators on generalized Morrey spaces,
Bound. Value Probl. 2009, Art. ID 835865, 18 pp.

\vspace{-.3cm}

\bibitem{s1938}
S. L. Sobolev, On a theorem in functional analysis,
Mat. Sb. 46 (1938), 471-497.

\vspace{-.3cm}

\bibitem{EMS1970}
E. M. Stein, Singular Integrals and Differentiability
Properties of Functions, Princeton Mathematical Series 30,
Princeton University Press, Princeton, NJ, 1970.

\vspace{-0.3cm}

\bibitem{sw1960}
E. M. Stein and G. Weiss, On the theory of harmonic
functions of several variables. I. The
theory of $H^p$-spaces, Acta Math. 103 (1960), 25-62.

\vspace{-0.3cm}

\bibitem{syy}
J. Sun, D. Yang and W. Yuan,
Weak Hardy spaces associated with ball quasi-Banach
function spaces on spaces of homogeneous type:
decompositions, real interpolation, and
Calder\'on--Zygmund operators,
J. Geom. Anal. 32 (2022), Paper No. 191, 85 pp.

\vspace{-0.3cm}

\bibitem{syy2}
J. Sun, D. Yang and W. Yuan,
Molecular characterization of weak Hardy spaces associated with
ball quasi-Banach function spaces on spaces of homogeneous type
with its application to Littelwood--Paley function characterization,
Submitted.

\vspace{-0.3cm}

\bibitem{tw80}
M. H. Taibleson and G. Weiss,
The molecular characterization of certain Hardy spaces,
Representation theorems for Hardy spaces,
in: Ast\'{e}risque 77, 67-149, Soc. Math. France, Paris, 1980.

\vspace{-0.3cm}

\bibitem{Tan2020}
J. Tan, Boundedness of multilinear fractional type operators on Hardy
spaces with variable exponents,
Anal. Math. Phys. 10 (2020), Paper No. 70, 16 pp.

\vspace{-0.3cm}

\bibitem{tx2005}
L. Tang and J. Xu, Some properties of Morrey type Besov-Triebel spaces, Math. Nachr.
278 (2005), 904-917.

\vspace{-0.3cm}

\bibitem{tyyz}
J. Tao, D. Yang, W. Yuan and Y. Zhang, Compactness
characterizations of commutators
on ball Banach function spaces, Potential Analysis (2021),
https://doi.org/10.1007/s11118-021-09953-w.

\vspace{-0.3cm}

\bibitem{wyy}
F. Wang, D. Yang and S. Yang, Applications of Hardy spaces
associated with ball quasi-Banach function spaces,
Results Math. 75 (2020), Paper No. 26, 58 pp.

\vspace{-0.3cm}

\bibitem{wyyz}
S. Wang, D. Yang, W. Yuan and Y. Zhang,
Weak Hardy-type spaces associated with ball quasi-Banach function
spaces II: Littlewood--Paley characterizations
and real interpolation,
J. Geom. Anal. 31 (2021), 631-696.

\vspace{-0.3cm}

\bibitem{yhyy1}
X. Yan, Z. He, D. Yang and W. Yuan, Hardy spaces associated
with ball quasi-Banach function spaces on spaces of homogeneous
type: Characterizations of maximal functions, decompositions,
and dual spaces, Math. Nachr. (2022),
http://doi.org/10.1002/mana.202100432.

\vspace{-0.3cm}

\bibitem{yhyy}
X. Yan, Z. He, D. Yang and W. Yuan, Hardy spaces associated with
ball quasi-Banach function spaces on spaces of homogeneous type:
Littlewood--Paley characterizations with applications to boundedness
of Calder\'on--Zygmund operators, Acta Math. Sin. (Engl. Ser.)
(2022), http://doi.org/10.1007/s10114-022-1573-9.

\vspace{-0.3cm}

\bibitem{yyy20}
X. Yan, D. Yang and W. Yuan,
Intrinsic square function characterizations
of Hardy spaces associated with ball quasi-Banach function spaces,
Front. Math. China 15 (2020), 769-806.

\vspace{-0.3cm}

\bibitem{YSY2010}
W. Yuan, W. Sickel and D. Yang, Morrey and Campanato meet Besov,
Lizorkin and Triebel,
Lecture Notes in Mathematics 2005. Springer-Verlag, Berlin, 2010.

\vspace{-0.3cm}

\bibitem{ZZ2022}
H. Zhang and J. Zhou,
The boundedness of fractional integral operators in local
and global mixed Morrey-type spaces,
Positivity 26 (2022), Paper No. 26, 22 pp.

\vspace{-0.3cm}

\bibitem{zhyy2022}
Y. Zhang, L. Huang, D. Yang and W. Yuan,
New ball Campanato-type function spaces and their applications,
J. Geom. Anal. 32 (2022), Paper No. 99, 42 pp.

\vspace{-0.3cm}

\bibitem{zyyw}
Y. Zhang, D. Yang, W. Yuan and S. Wang,
Weak Hardy-type spaces associated with ball quasi-Banach function
spaces I: Decompositions with applications to boundedness of
Calder\'on--Zygmund operators,
Sci. China Math. 64 (2021), 2007-2064.

\vspace{-0.3cm}

\bibitem{zyz2022}
Y. Zhao, D. Yang and Y. Zhang,
Mixed-norm Herz spaces and their applications in related Hardy
spaces, Submitted or arXiv: 2204.12019v1.

\end{thebibliography}
\end{document}